\documentclass[12pt,oneside,english]{amsart}
\usepackage[T1]{fontenc}
\usepackage[utf8]{inputenc}
\usepackage{geometry}
\usepackage{fancyhdr}
\usepackage{color}
\usepackage[shortlabels]{enumitem}
\usepackage{amsthm}
\usepackage{amstext}
\usepackage{amssymb}
\usepackage{mathdots}
\usepackage{esint}
\usepackage[final]{graphicx}
\usepackage{subcaption}

\makeatletter
\numberwithin{equation}{section}
\numberwithin{figure}{section}
\theoremstyle{plain}
\newtheorem{thm}{\protect\theoremname}
\theoremstyle{remark}
\newtheorem*{rem*}{\protect\remarkname}
\theoremstyle{plain}
\newtheorem{rem}{\protect\remarkname}
\theoremstyle{plain}
\newtheorem{lem}[thm]{\protect\lemmaname}
\theoremstyle{plain}
\newtheorem{prop}[thm]{\protect\propositionname}
\theoremstyle{plain}

\theoremstyle{plain}

\theoremstyle{plain}

\usepackage[english]{babel}
 \usepackage{fancyhdr}% http://ctan.org/pkg/fancyhdr
\providecommand{\lemmaname}{Lemma}
\providecommand{\propositionname}{Proposition}
\providecommand{\theoremname}{Theorem}
\providecommand{\corollaryname}{Corollary}
\providecommand{\remarkname}{Remark}
\newcommand{\abs}[1]{\ensuremath{|#1|}}

\newcommand{\Abs}[1]{\ensuremath{\left|#1\right|}}

\newcommand{\norm}[2]{\ensuremath{|\!|#1|\!|_{#2}}}

\newcommand{\Norm}[2]{\ensuremath{\left|\!\left|#1\right|\!\right|_{#2}}}

\newcommand{\braket}[2]{\langle #1 | #2 \rangle}

\renewcommand{\d}[1]{\ensuremath{\textnormal{d}#1}}

\newcommand{\cC}{\mathcal{C}}

\newcommand{\cE}{\mathcal{E}}

\newcommand{\cI}{\mathcal{I}}

\newcommand{\cO}{\mathcal{O}}

\DeclareMathOperator\Log{Log}
\DeclareMathOperator\Arg{Arg}

\makeatother

\providecommand{\lemmaname}{Lemma}
\providecommand{\propositionname}{Proposition}
\providecommand{\remarkname}{Remark}
\providecommand{\theoremname}{Theorem}

\usepackage[ruled]{algorithm2e}
\usepackage{listings}
\usepackage{tikz}
\usetikzlibrary{arrows,shapes}
\usetikzlibrary{calc,decorations.markings}
\tikzset{
  reversed with radius/.style={
    x radius=#1,
    y radius=-#1,
 }
}

\tikzset{
  with arrows/.style={
    decoration={ markings,
      mark=between positions #1 and .999 step #1 with {\arrow{stealth}}
    }, postaction={decorate}
  }, with arrows/.default=25mm,
}

\begin{document}

\title{On the asymptotic behavior of Jacobi polynomials with first varying
parameter}

\author{Oleg Szehr}

\address{Centre for Quantum Information, University of Cambridge, Cambridge
CB3 0WA, United Kingdom.}

\address{Dalle Molle Institute for Artificial Intelligence (IDSIA) - SUPSI/USI,
Manno, Switzerland}

\email{oleg.szehr@idsia.ch}

\author{Rachid Zarouf}

\address{Aix-Marseille University, ADEF, University campus of Saint-J\'er\^ome, 52 Avenue
Escadrille Normandie Niemen, 13013 Marseille}

%\address{Institut de Mathématiques de Marseille, UMR 7373, Aix-Marseille Université,
%39 rue F. Joliot Curie, 13453 Marseille Cedex 13, France.}

\email{rachid.zarouf@univ-amu.fr}

\address{Department of Mathematics and Mechanics, Saint Petersburg State University,
28, Universitetski pr., St. Petersburg, 198504, Russia.}

\thanks{The work is supported by the project ANR 18-CE40-0035}

\email{rzarouf@gmail.com}

\begin{abstract}
We investigate the large $n$ behavior of Jacobi polynomials with varying parameters $P_{n}^{(an+\alpha,\,bn+\beta)}(1-2\lambda^{2})$ for $a,b >-1$ and $\lambda\in(0,\,1)$. This is a well-studied topic in the literature but some of the published results appear to be discordant. To address this issue we provide an in-depth investigation of the case $b = 0$, which is most relevant for our applications. Our approach is based on a new and surprisingly simple representation
of $P_{n}^{(an+\alpha,\,\beta)}(1-2\lambda^{2}),\:a>-1$ in terms
of two integrals. The integrals' asymptotic behavior is studied using
standard tools of asymptotic analysis: one is a Laplace integral and
the other is treated via the method of stationary phase. As a consequence
we prove that if $a\in(\frac{2\lambda}{1-\lambda},\infty)$ then $\lambda^{an}P_{n}^{(an+\alpha,\beta)}(1-2\lambda^{2})$
shows exponential decay and we derive simple exponential upper bounds
in this region.~If~$a\in(\frac{-2\lambda}{1+\lambda},\,\frac{2\lambda}{1-\lambda})$
then the decay of $\lambda^{an}P_{n}^{(an+\alpha,\beta)}(1-2\lambda^{2})$
is $\cO(n^{-1/2})$ and if $a\in\{\frac{-2\lambda}{1+\lambda},\,\frac{2\lambda}{1-\lambda}\}$
then $\lambda^{an}P_{n}^{(an+\alpha,\beta)}(1-2\lambda^{2})$ decays
as $\cO(n^{-1/3})$.~A~new phenomenon occurs in the parameter range
$a\in(-1,\frac{-2\lambda}{1+\lambda})$, where we find that the behavior
depends on whether or not $an+\alpha$ is an integer: If~$a\in(-1,\frac{-2\lambda}{1+\lambda})$
and $an+\alpha$ is an integer then $\lambda^{an}P_{n}^{(an+\alpha,\beta)}(1-2\lambda^{2})$
decays exponentially.~If $a\in(-1,\frac{-2\lambda}{1+\lambda})$
and $an+\alpha$ is not an integer then $\lambda^{an}P_{n}^{(an+\alpha,\beta)}(1-2\lambda^{2})$
may increase exponentially depending on the proximity of the sequence $(an + \alpha)_n$ to integers.

\end{abstract}

\maketitle

\section{\label{sub:Introduction}Introduction}

The Jacobi polynomials $P_{n}^{(\alpha,\,\beta)}$ (occasionally also
called hypergeometric polynomials) constitute a wide class of classical polynomials.
They include the Gegenbauer and thus also the Legendre, Zernike and
Chebyshev polynomials as special cases. For $\alpha,\beta\in\mathbb{R}$
and $x\in(-1,1)$ the Jacobi polynomials of degree $n$ can be defined~\cite[p.~68]{GZ} by
\begin{align*}
P_{n}^{(\alpha,\,\beta)}(x)=\sum_{\mu=0}^{n}\binom{n+\alpha}{n-\mu}\binom{n+\beta}{\mu}\left(\frac{x-1}{2}\right)^{\mu}\left(\frac{x+1}{2}\right)^{n-\mu}.
\end{align*}
This article is concerned with the approximation of large degree Jacobi
polynomials. While for fixed parameters $\alpha,\beta$ this is a classical
topic of asymptotic analysis and approximation theory~\cite{WW,GZ}, this work is concerned with the more general Jacobi polynomials with varying parameters (JPVPs)
\[
P_{n}^{(an+\alpha,bn+\beta)}(x),\qquad{a,b,\alpha,\beta>-1}.
\]
The asymptotic behavior of JPVPs has been the content of multiple preceeding publications including~\cite{BG,CI,FFN,GS,SI,SV,CIR}. We have observed that some of the results are discordant, which led us to search for a simple and robust approach to the topic. The articles~\cite{CIR,SI,SV} investigate the case $b = 0$, i.e.~only the first parameter of JPVPs depends on $n$. We call such polynomials Jacobi polynomials with first varying parameter (JPFVPs). Our presentation
is streamlined to this case because it is the most relevant for our applications and to allow for an immediate comparison to~\cite{CIR,SI,SV}.
At the core
of our investigation lies a new representation of Jacobi polynomials in terms
of two integrals. The representation is chosen such that the integrals'
asymptotic behavior can be determined by means of well-established
methods. The first integral is a generalized Fourier integral, whose
asymptotic behavior follows from Erdélyi's classical method of stationary
phase~\cite{AE}. The second integral is of Laplace type: its asymptotic
behavior is determined by the classical method of Laplace~\cite[Section 6.4]{BO}.
Thus we introduce a new and comparatively simple approach to the topic.
Based on our integral representation we systematically determine
the asymptotic behavior of JPFVPs and we explain how
our method extends to the general case of JPVPs. The theoretical investigation is accompanied with extensive numerical experiments.

We find that the behavior of JPFVPs is non-trivial in the sense that different types of asymptotics occur as $n\rightarrow\infty$. To state the results it is convenient to change the variable of the polynomials writing $x=1-2\lambda^2$ with $\lambda\in(0,1)$. As we will prove
in the main body of this article there are four distinct parameter regions, 1)~$a\in(-1,\frac{-2\lambda}{1+\lambda})$,
2)~$a\in(-\frac{2\lambda}{1+\lambda},\,\frac{2\lambda}{1-\lambda})$,
3) $a\in(\frac{2\lambda}{1-\lambda},\infty)$ and 4) $a\in\left\{ -\frac{2\lambda}{1+\lambda},\,\frac{2\lambda}{1-\lambda}\right\} $
with categorically different asymptotics. The below listing summarizes
our findings in the respective cases and compares them to the literature. 
\begin{enumerate}
\item The case $a\in(-1,\frac{2\lambda}{1-\lambda})$ has been studied in
\cite{CI} and \cite{SI}. According to \cite{SI} the results of
\cite{CI} must be corrected. \cite{SI} claims that if $a\in(-1,\frac{2\lambda}{1-\lambda})$
then the decay of $\lambda^{an}P_{n}^{(an+\alpha,\,\beta)}(1-2\lambda^{2})$
is $\cO\left(n^{-1/2}\right)$. Our result, see Theorem~\ref{gbar}
below, contradicts both \cite{CI} and \cite{SI}. For $a\in(-1,-\frac{2\lambda}{1+\lambda})$
we identify a new form of asymptotic behavior: 
\begin{enumerate}
\item If $an+\alpha$ is an integer then $\lambda^{an}P_{n}^{(an+\alpha,\,\beta)}(1-2\lambda^{2})$
decays exponentially in magnitude. 
\item If $an+\alpha$  remains strictly separate from integers on a subsequence of natural numbers then
\[
\lambda^{an}P_{n}^{(an+\alpha,\beta)}(1-2\lambda^{2})
\]
increases exponentially in magnitude along this subsequence.

%is not an integer then 

%\[
%\textcolor{red}{\lambda^{an}P_{n}^{(an+\alpha,\beta)}(1-2\lambda^{2})}
%\]
%\[
%\frac{\lambda^{an}P_{n}^{(an+\alpha,\beta)}(1-2\lambda^{2})}{\sin{(\pi(an+\alpha))}}
%\]
%
%
%\textcolor{red}{may increase} exponentially in magnitude, \textcolor{red}{depending on the proximity of the sequence $(an+\alpha)_n$ to integers}.
%
%
%
\end{enumerate}
\item It is shown in~\cite{GS,FFN} that if $a\in\left[0,\frac{2\lambda}{1-\lambda}\right)$
then $\lambda^{an}P_{n}^{(an+\alpha,\beta)}(1-2\lambda^{2})$ decays
in first order like $\cO\left(n^{-1/2}\right)$. The wider interval
$a\in(-\frac{2\lambda}{1+\lambda},\,\frac{2\lambda}{1-\lambda})$
is studied in~\cite{SI,CI}, where a decay of $\cO\left(n^{-1/2}\right)$
is found. However, the explicit formulas of~\cite{SI,CI} and~\cite{GS,FFN}
are inconsistent. In this range of parameters our findings are identical
to those of \cite[Theorem 1]{GS} and to \cite[Proposition 6.1]{FFN} (but notice the different $\cO$-term, see Remark~\ref{oscilatingRemark}). 
\item For $a>\frac{2\lambda}{1-\lambda}$ we find that $\lambda^{an}P_{n}^{(an+\alpha,\,\beta)}(1-2\lambda^{2})$
decays exponentially which is in line with \cite{SI,CI}, see Theorem
\ref{gbar} points (3) and (4) for details. 
\item For $a\in\left\{ -\frac{2\lambda}{1+\lambda},\,\frac{2\lambda}{1-\lambda}\right\} $
we find that $\lambda^{an}P_{n}^{(an+\alpha,\,\beta)}(1-2\lambda^{2})$
decays in first order like $\cO(n^{-1/3})$. Our asymptotic formulas,
see Theorem~\ref{gbar} point (2) for details, differ from \cite[ Formula (3.9)]{CI}.
This case is not investigated in \cite{SI}. An asymptotic formula is also derived in~\cite{SV} for $a=\frac{2\lambda}{1-\lambda}$ but the printed formula differs from ours by a complex prefactor.
\end{enumerate}
Our results cannot be directly compared to those obtained in~\cite{KM,MO}
because the authors are interested in a different range of parameters.
\cite{KM}~focuses on parameters that satisfy the following two
conditions: $-1<a<0$, $-1<b<0$, $a+b<-1$ and $\alpha=\beta=0$. \cite{MO} treats the case that $a,\,b$
are not both positive and take values outside the triangle bounded by $a=0,\,b = 0$ and $a+b+2=0$ and $\alpha=\beta=0$.
But for negative values of $a, b$ their formulas witness a similar dependency on the proximity of $(a+b)n$ to an integer as the one we find in region 1), see for example~\cite[Theorem 2.6]{KM} and~\cite[Remark 2.9]{KM}.
Our interest in the asymptotic behavior of Jacobi polynomials arose from their relation to the Fourier coefficients of the Blaschke product
$B=b^{n}=\left(\frac{z-\lambda}{1-\lambda z}\right)^{n}$. In fact the $k^{th}$ Fourier coefficient of $B$ is given by
\begin{align}
\hat{B}(k)=\lambda^{k-n}(1-\lambda^{2})P_{n-1}^{(k-n,\,1)}(1-2\lambda^{2}),\label{diskAuto}
\end{align}
which is a consequence of Lemma~\ref{striking} below. Being the $n$-th power of the disk automorphism, $B$ plays a key role in the operator theory of holomorphic Banach spaces of the unit disk. Consider,
for instance, the classical Beurling-Sobolev space 
\[
l_{p}^{A}(w):=\left\{ f=\sum_{k\geq0}\hat{f}(k)z^{k}\in Hol(\mathbb{D})\:\Big|\:\Norm{f}{l_{p}^{A}(w)}:=\left(\sum_{k\geq0}\abs{\hat{f}(k)}^{p}w_{k}^{p}\right)^{1/p}<\infty\right\} ,
\]
where the sequence of weights $w=(w_{k})_{k\geq1}$ is such that $w_{k}>0$,
$\lim_{k}w_{k}^{1/k}=1$ \cite[p.  670]{NN2}. We write briefly $l_p^A$ for the special case of the analytic sequence space, $w=(1)_{k\geq1}$. The asymptotic $n$-dependency of $\Norm{B}{l_{p}^{A}(w)}$ is connected
to the asymptotics of JPFVPs via Equation~\eqref{diskAuto} because the asymptotic analysis requires the investigation of regions of Fourier coefficients where $k$ is proportional to $n$. The regions 1), 2), 3), 4) typically play different roles
in terms of their contribution to $\Norm{B}{l_{p}^{A}(w)}$. The region of exponential
decay of $\abs{\hat{B}(k)}$ adds an asymptotically negligible contribution. Simple exponential upper estimates are a therefore a useful tool to restrict the analysis to the relevant coefficients. In general, it will depend on $w$ and $p$, which region contributes the terms that dominate the asymptotic behavior.

We illustrate the relevance of information about the asymptotic behavior of $\Norm{B}{l_{p}^{A}(w)}$ by some examples. First, the asymptotic behavior of $\Norm{B}{l_{p}^{A}}$, $p\in[1,2]$, has been determined in~\cite{BS} to study the boundedness of composition operators
$(comp_{b}(f))(z):=f(b(z))$ on $l_{p}^{A}$. In similar spirit, information about $\Norm{B}{l_{p}^{A}\left(w\right)}$ has recently been exploited in~\cite{LLQRP} to identify the sequences $w=(w_{k})_{k\geq1}$,
for which the composition operator on the weighted Hardy space
$l_{2}^{A}(w)$ is bounded. We have computed the exact asymptotic $n$-dependency of $\Norm{B}{l_{p}^{A}}$ for $p\in[1,\infty]$ in~\cite{SZ2}. Second, $\Norm{B}{l_{p}^{A}(w)}$
plays an important role in interpolation theory: Consider the formal Hölder inequality
\[
\abs{\braket{f}{g}}\leq\Norm{f}{l_{q}^{A}}\Norm{g}{l_{p}^{A}}
\]
where $1/p+1/q=1$ and $\braket{\cdot}{\cdot}$ denotes the usual scalar product on $L^{2}(\partial\mathbb{D})$. 
%Then any $f\in l_{q}^{A}$, $g\in l_{p}^{A}$ satisfy
%\[
%\Norm{f}{l_{q}^{A}}\geq\frac{\abs{\braket{f}{g}}}{\Norm{g}{l_{p}^{A}}}.
%\]
For the Nevanlinna-Pick-type interpolation problem \cite{NN2} in $l_q^A$ with $(\lambda_i)_{i=1}^n\subset\mathbb{D}$,
\[
\cI=\inf\{\Norm{f}{l_{q}^{A}}: f(0)=1,\ f(\lambda_i)=0,\ i=1,...,n\},
\]
choosing $g=B$, where $B$ is the Blaschke product associated to the sequence $(\lambda_i)_{i=1}^n$,  entails the lower estimate
\[
\cI\geq\frac{1}{\Pi_{i=1}^n\abs{\lambda_i}\Norm{B}{l_{p}^{A}}}.
\]
%where here $B$ is the Blaschke product associated to the sequence $(\lambda_i)_{i=1}^n$.
%
In other words any $f$ that is admissible for $\cI$ grows at least as quickly as $1/({\Pi_{i=1}^n\abs{\lambda_i}\Norm{B}{l_{p}^{A}}})$. The same reasoning extends mutatis mutandis to the more general Beurling-Sobolev spaces $l_{p}^{A}(w)$.
Third, the norm $\Norm{B}{l_{p}^{A}(w)}$ occurs in extremal problems in matrix analysis. While an introduction to model theory would lead beyond the scope of this article (see~\cite{NN1} for details) we briefly illustrate the flavor of problems in this area. Suppose $A$ is a complex $n\times n$ matrix with spectrum $(\lambda_i)_{i=1}^n$, $\Norm{\cdot}{}$ an induced matrix norm with $\Norm{A}{}\leq1$, and let $\psi$ be any rational function, whose poles are disjoint from the eigenvalues of $A$. Consider the extremal problem
\begin{align*}
\cE=\sup\Norm{\psi(T)}{},
\end{align*}
where the supremum goes over all $T$ with $\Norm{T}{}\leq1$ that have the same minimal polynomial as $A$. Then it holds~\cite{NN2,SZ1}
\begin{align}\cI = \inf\{\norm{f}{l_1^A}\: |\:f(\lambda_i)=\psi(\lambda_i), i=1,...,n\}=\Norm{\psi(M)}{}\leq\cE,\label{extremal}
\end{align}
%
%
%\begin{align}\Norm{\psi(A)}{}\leq\Norm{\psi(M)}{},\label{extremal}\end{align}
where $M$ is an operator model of $A$: $M$ has the same minimal polynomial as $A$ and $\Norm{M}{}\leq1$, too. Thus, the techniques to bound $\cI$ from below are also available for the estimation of $\cE$. This method has
been employed (for $\psi(z)=z^{-1}$) to introduce a constructive approach to a conjecture
of Schäffer~\cite{SZ1}. In similar vein, there is an open question of V. Pták \cite{You1,Pt2}
from the 1980's: Over all $n\times n$ complex matrices $A$ with
$\Norm{A}{2}\leq1$, where $\Norm{\cdot}{2}$ is the Hilbert space operator norm,
and with spectral radius bounded by $\lambda\in(0,1)$ determine the supremum of $\Norm{A^{n}}{2}$
as a function of $n$ and $\lambda$. It
is shown in \cite{Pt2} that the supremum is attained by the model operator whose spectrum $(\lambda_{i})_{i=1}^{n}$ is
fully degenerate at point $\lambda$.
As it turns out the entries of the $N$-th power of the model operator can be expressed (in Malmquist-Walsh basis \cite[p. 117]{NN1}) as
\begin{align*}
\left(M^{N}\right)_{i,j} = \begin{cases} (-\lambda)^{i-j-N}(1-\lambda^{2})P_{N-1}^{(i-j-N,1)}(1-2\lambda^{2}),\ &i> j,\\
\lambda^N,\ &i=j,\\
0,\ &i<j,
\end{cases}
\end{align*}
which relates Pták's question to the theory of JPFVPs by~\eqref{extremal}.

Section~\ref{sec:Main-result} contains our main results. Our insights
about the asymptotic behavior of JPFVPs are presented in Section~\ref{asylanticAnalysis}
and our new integral representation for JPFVPs in Section~\ref{newIntRep}.
Section~\ref{asymptInt} describes our methods to determine the asymptotic behavior of the involved integrals. The generalized Fourier
integral is treated in Section~\ref{sub:The-stationary-phase} and
the Laplace integral in Section~\ref{sub:Laplace's-method-for}.
Section~\ref{sec:Simple-exponential-bounds} provides simple exponential
estimates for the JPFVPs in the region of exponential decay/ growth. The proof of our main theorem is given in Section~\ref{proofSection}. Section~\ref{numericalExperiments}
contains the results of our numerical experiments. Building on our integral representation Section~\ref{sec:Final-remarks}
describes possible generalizations of our work. This includes the
case of JPVPs with $b\neq0$ and a description of the methodology to obtain
uniform asymptotic expansions as $a$ approaches the boundaries.

\section*{acknowledgments}

The authors are grateful to A. D. Baranov, A. A. Borichev, S. Charpentier, S. Kupin for their comments on an earlier version of our manuscript.

\section{Main results}

\label{sec:Main-result} 

\subsection{On the asymptotic behavior of Jacobi polynomials with first varying parameter}

\label{asylanticAnalysis} 
\begin{thm}
\label{gbar} Let $\alpha,\ \beta>-1$, $a>-1$ and $\lambda\in(0,\,1)$.
The following asymptotic formulas hold as $n\rightarrow\infty$: 
\begin{enumerate}
\item If $a\in(-\frac{2\lambda}{1+\lambda},\,\frac{2\lambda}{1-\lambda})$
then 
\begin{align*}
 & \lambda^{an+\alpha}P_{n}^{(an+\alpha,\,\beta)}(1-2\lambda^{2})=\sqrt{\frac{2}{n\pi}}\frac{((1-\lambda^{2})(a+1))^{-\frac{\beta}{2}}}{((1-\lambda^{2})((a+2)\lambda+a)((a+2)\lambda-a))^{\frac{1}{4}}}\\
 & \cdot\cos{\left((n+1)h(\varphi_{+})+(\alpha-a)\varphi_{+}+(\beta-1)\psi{+}\frac{\pi}{4}\right)}+ \cO\left(\frac{1}{n^{3/2}}\right).
\end{align*}
The phases $\varphi_{+},h(\varphi_{+}),\psi\in(-\pi,\pi]$ depend on
$a,\lambda$ and are given explicitly in Proposition~\ref{thm:delta_negative} below.
\item If $a=\frac{2\lambda}{1-\lambda}$ then 
\begin{align*}
\lambda^{an+\alpha}P_{n}^{(an+\alpha,\,\beta)}(1-2\lambda^{2})=\frac{(1+\lambda)^{-\beta}}{3^{2/3}\Gamma(2/3)n^{1/3}\lambda^{1/3}(1+\lambda)^{1/3}}\left(1+\cO\left(\frac{1}{n^{1/3}}\right)\right).
\end{align*}
If $a=-\frac{2\lambda}{1+\lambda}$ then 
\begin{align*}
 & \lambda^{an+\alpha}P_{n}^{(an+\alpha,\,\beta)}(1-2\lambda^{2})=\frac{(1-\lambda)^{-\beta}}{3^{2/3}\Gamma(2/3)\pi n^{1/3}\lambda^{1/3}(1-\lambda)^{1/3}}\\
 & \cdot\left(\cos\left((an+\alpha)\pi\right)-\sqrt{3}\sin\left((an+\alpha)\pi\right)\right)+ \cO\left(\frac{1}{n^{2/3}}\right).
\end{align*}
\item If $a\in(-1,\frac{-2\lambda}{1+\lambda})\cup(\frac{2\lambda}{1-\lambda},\infty)$
and $an+\alpha$ is an integer then the quantity 
\[
\lambda^{an+\alpha}P_{n}^{(an+\alpha,\,\beta)}(1-2\lambda^{2})
\]
decays exponentially in magnitude with $n$, see Proposition~\ref{Second_thm}
for an upper bound.
\item 
If $a\in(-1,-\frac{2\lambda}{1+\lambda})$ and if the sequence $(an+\alpha)_n$ is separate from integers then the quantity 
\[
\lambda^{an+\alpha}P_{n}^{(an+\alpha,\,\beta)}(1-2\lambda^{2})
\]
increases exponentially in magnitude. The precise behavior depends on the proximity of the sequence $(an+\alpha)_n$ to integers, see the proof of Theorem~\ref{gbar} point (4). 
If $a\in(\frac{2\lambda}{1-\lambda},\infty)$ and $an+\alpha$ is not an integer then the above quantity decays
exponentially in magnitude.
\end{enumerate}
\end{thm}
\begin{rem}
\label{Rk1}
Our approach to prove points 1), 2) can also be used to derive asymptotic formulas in the exponential region $a\in(-1,\frac{-2\lambda}{1+\lambda})\cup(\frac{2\lambda}{1-\lambda},\infty)$, see Section~\ref{proofSection}. We prefer the shorter formulation here, because it is sufficient for our application.
\end{rem}
\begin{rem}
The developed methodology yields an expansion
of $P_{n}^{(an+\alpha,\,\beta)}(1-2\lambda^{2})$ to any order. The required computations are straight-forward albeit tedious. 
\end{rem}
\begin{rem}\label{oscilatingRemark}
Determining the \emph{precise} asymptotic growth of the JPVFPs can be a delicate task in cases, where the formulas involve oscillating trigonometric functions of $n$. To illustrate the phenomenon consider the oscillating term
$$\sqrt{\frac{2}{n\pi}}\cos\left(((n+1)h(\varphi_{+})+(\alpha-a)\varphi_{+}+(\beta-1)\psi{+}\frac{\pi}{4}\right)$$
in 1), which is clearly $\cO(1/\sqrt{n})$. While the second summand in 1) is only $\cO(1/\sqrt{n^{3}})$ it is still not obvious, which term actually provides the dominating contribution to the JPFVPs. For completeness we show in Section~\ref{proofSection} that given $\gamma\in\mathbb{R}$ there exists a subsequence $s_n\in\mathbb{N}$ such that $\abs{\cos(\gamma s_n+\frac{\pi}{2})}\leq K/s_n$ with a constant $K$. Notice that along such a subsequence the oscillating term can be $\cO(1/\sqrt{n^3})$.
\end{rem}
The theorem is an immediate consequence of our representation of the JPFVPs in terms of two integrals, Lemma~{\ref{striking}}, together with the
asymptotic formulas for the individual integrals, Proposition~\ref{thm:delta_negative}
and Proposition~\ref{Laplace}. The results are collected in Section~\ref{proofSection}.
Various methods are employed in the literature to prove asymptotic
expansions. Gawronksi-Shawyer~\cite{GS}~as well as Saff-Varga~\cite{SV}
rely on the method of steepest descent~\cite[p. 147]{BO}
\cite[p. 136]{FO0}, while Chen-Ismail~\cite{CI} and Izen~\cite{SI}
make use of Darboux's asymptotic method and generating functions~\cite{FO0,GZ}.
In this article classical methods of asymptotic analysis are employed to determine the asymptotic behavior of our integral representation. Even in situations when no integral representation
is available, the Riemann-Hilbert approach~\cite{KM,MO} can still
be employed to derive asymptotic expansions. The Riemann-Hilbert approach
has the additional advantage that it automatically provides uniform
(over $a,b,\lambda$) asymptotic expansions. A uniform expansion can
also be obtained from our integral representation using well-established
methods, see Section~\ref{uniform}, although this is not
the main focus of the article. 

\subsection{A new integral representation for Jacobi Polynomials}

\label{newIntRep} Although simple, we regard the below lemma as the
main innovation of this work. In what follows powers
are defined with respect to the principal branch of the complex logarithm,
which is denoted by $\Log$. 
\begin{lem}[Integral representation for JPFVPs]
\label{striking} Let $n$ be an integer,{{} let
$a>-1$ and $\alpha,\,\beta\in\mathbb{R}$ be such that $(a+1)n+\alpha>-1$.
}Given $\lambda\in(0,\,1)$, for any $x\in(\lambda,\,1/\lambda)$
we have the following integral representation for the Jacobi polynomials
with first varying parameter
\begin{align*}
 & \lambda^{an+\alpha}(1-\lambda^{2})^{\beta}P_{n}^{(an+\alpha,\,\beta)}(1-2\lambda^{2})\\
 & =\frac{1}{\pi}\Re\left\{ \int_{0}^{\pi}{z^{\alpha+1}\frac{(1-\lambda z)^{\beta}}{z-\lambda}\bigg(z^{a+1}\frac{(1-\lambda z)}{z-\lambda}\bigg)^{n}}\Bigg|_{z={x}e^{i\varphi}}\d\varphi\right\} \\
 & -\frac{{\sin\left(\pi(an+\alpha)\right)}}{\pi}\int_{0}^{{x}}\frac{(1+\lambda t)^{\beta}t^{\alpha}}{t+\lambda}\left(t^{(a+1)}\frac{(1+\lambda t)}{t+\lambda}\right)^{n}{\rm d}t,
\end{align*}

where $\Re(\bullet)$ stands for the real part of a complex number. 
\end{lem}

This representation
is particularly suited for asymptotic analysis.

\begin{enumerate}
\item The first integral can be written as a so-called \textit{generalized
Fourier integral} \cite{AE}, \cite[Chapter 6.5]{BO}. The latter
is of the form 
\[
\int g(\varphi)e^{inh(\varphi)}\d\varphi
\]
with continuous and real functions $g$, $h$. To determine the asymptotic behavior we employ the method of stationary phase, which is
commonly used to study asymptotic properties of oscillatory integrals.
The method is conceptually similar to \emph{Laplace's method} (see
below) in that the leading order terms are contributed by a small neighborhood
around the stationary points of $h$. Our analysis relies on the well-known
fact~\cite[Theorem 4]{AE},~\cite[Chapter 6.5]{BO} that if $\xi$
is a stationary point of $h$ such that $g(\xi)\neq0$ then 
\begin{enumerate}
\item the integral goes to 0 as $n^{-1/2}$ if $h''(\xi)\neq0$, 
\item the integral goes to 0 as $n^{-1/3}$ if $h''(\xi)=0$ but $h^{(3)}(\xi)\neq0$. 
\end{enumerate}
\item The second integral contributes iff $an+\alpha$ is not an integer
and is of~\emph{Laplace type}~\cite[Chapter 3]{FO0}. The latter
is of the form 
\[
\int g(t)e^{nh(t)}\d t,
\]
with continuous and real functions $g$, $h$. To determine the asymptotic
behavior we employ Laplace's method. The idea is that if $h$ has
a maximum at $\xi$ and if $g(\xi)\neq0$ then as $n$ grows only
values in an immediate neighborhood of $\xi$ contribute~\cite{BO,FO0}. 
\end{enumerate}
Curiously the second integral contributes iff $an+\alpha$ is not
an integer. As a consequence the magnitude of JPFVPs \lq\lq{}scatters\rq\rq{}
as a function of $n$ depending on $(a,\alpha)$.
To determine the asymptotic behavior it is therefore necessary to
consider subsequences of $n$. On a subsequence of $n$ where $an+\alpha$
is an integer the second integral does never contribute. However for
such $n$ that $an+\alpha$ is not an integer the second integral
contributes, which potentially alters the asymptotic behavior in a
\lq\lq{}discontinuous\rq\rq{} way. 
\begin{proof}
We investigate the complex
function 
\[
f:z\mapsto z^{(a+1)n+\alpha}\frac{(1-\lambda z)^{n+\beta}}{(z-\lambda)^{n+1}}.
\]
We integrate $f$ along a closed contour $\gamma$ in the complex
plane, see Figure~\ref{fig:contour1}. The contour is composed of
four curves $\gamma=\gamma_{1}\oplus\gamma_{2}\oplus\gamma_{3}\oplus\gamma_{4}$
with $x\in(\lambda,1/\lambda)$ and 
\begin{align*}
\gamma_{1} & :\varphi\mapsto xe^{i\varphi},\quad\varphi\in(-\pi,\pi],\\
\gamma_{2} & :\varphi\mapsto\epsilon e^{-i\varphi},\quad\varphi\in(-\pi,\pi],\\
\gamma_{3} & :t\mapsto t+i\tilde{\epsilon},\quad t\in[-x,{0}],\\
\gamma_{4} & :t\mapsto-t-i\tilde{\epsilon},\quad t\in[0,{x}].
\end{align*}
% I start drawing the contour now
\begin{figure}
\begin{tikzpicture}[xscale=-1]
% Configurable parameters
\def\gap{0.2}
\def\bigradius{3}
\def\littleradius{0.5}
% Axes
\draw (-1.1*\bigradius, 0) -- (1.1*\bigradius,0)
       (0, -1.1*\bigradius) -- (0, 1.1*\bigradius);

         \draw[with arrows]
  (-1.7,0) circle[reversed with radius=0.4cm]
  ;
% path
\draw[black, thick,   decoration={ markings,
      mark=at position 0.05 with {\arrow[line width=1pt]{<}},
       mark=at position 0.1 with {\arrow[line width=1pt]{<}},
      mark=at position 0.17 with {\arrow[line width=1pt]{<}},
       mark=at position 0.22 with {\arrow[line width=1pt]{<}},
        mark=at position 0.3 with {\arrow[line width=1pt]{<}},
        mark=at position 0.41 with {\arrow[line width=1pt]{<}},
      mark=at position 0.47 with {\arrow[line width=1pt]{<}},
      mark=at position 0.53 with {\arrow[line width=1pt]{<}},
       mark=at position 0.60 with {\arrow[line width=1pt]{<}},
       mark=at position 0.68 with {\arrow[line width=1pt]{<}},
      mark=at position 0.755 with {\arrow[line width=1pt]{<}},
       mark=at position 0.81 with {\arrow[line width=1pt]{<}},
       mark=at position 0.84 with {\arrow[line width=1pt]{<}},
      mark=at position 0.87 with {\arrow[line width=1pt]{<}},
        mark=at position 0.9 with {\arrow[line width=1pt]{<}},
        mark=at position 0.955 with {\arrow[line width=1pt]{<}}
        mark=at position 0.01 with {\arrow[line width=1pt]{<}}},
      postaction={decorate}]
  let
     \n1 = {asin(\gap/2/\bigradius)},
     \n2 = {asin(\gap/2/\littleradius)}
  in (\n1:\bigradius) arc (\n1:360-\n1:\bigradius)
  -- (-\n2:\littleradius) arc (-\n2:-360+\n2:\littleradius)
  -- cycle;
% The labels
\node at (0.2,0.7) {$\gamma_{2}$};
\node at (-1.8,2.8) {$\gamma_{1}$};
\node at (1.555,0.5) {$\gamma_{3}$};
\node at (1.555,-0.5) {$\gamma_{4}$};
\node at (-1.7,0.7) {$\gamma_{5}$};
\node at (-1.7,-0.25) {$\lambda$};
\fill (-1.7,0) circle (0.05);
\node at (-3.1,-0.3) {$1$};
\fill (-3,0) circle (0.05);
\end{tikzpicture} \caption{Contours $\gamma=\gamma_{1}\oplus\gamma_{2}\oplus\gamma_{3}\oplus\gamma_{4}$
and $\gamma_{5}$ for $x=1$.}
\label{fig:contour1} 
\end{figure}
%
% END of  drawing the contour
Due to holomorphy we have that 
\begin{align*}
\frac{1}{2\pi i}\int_{\gamma_{1}\oplus\gamma_{2}\oplus\gamma_{3}\oplus\gamma_{4}}f\ \d z=\frac{1}{2\pi i}\int_{\gamma_{5}}f\ \d z,
\end{align*}
where $\gamma_{5}$ denotes a small circle around the pole $\lambda$,
see Figure~\ref{fig:contour1}, 
\begin{align*}
\gamma_{5}:\varphi\mapsto\lambda+se^{i\varphi},\ \varphi\in(-\pi,\pi].
\end{align*}
The lemma is proved by computing each of the five integrals individually. 
\begin{itemize}
\item We observe that 
\begin{align*}
\frac{1}{2\pi i}\int_{\gamma_{1}}f\ \d z & =\frac{1}{2\pi}\int_{-\pi}^{\pi}z^{(a+1)n+\alpha+1}\frac{(1-\lambda z)^{n+\beta}}{(z-\lambda)^{n+1}}\Bigg|_{z={x}e^{i\varphi}}\d\varphi\\
 & =\frac{1}{\pi}\Re\left\{ \int_{0}^{\pi}{z^{(a+1)n+\alpha+1}\frac{(1-\lambda z)^{n+\beta}}{(z-\lambda)^{n+1}}}\Bigg|_{z={x}e^{i\varphi}}\d\varphi\right\} .
\end{align*}
\item The standard estimate for contour integrals reads 
\begin{align*}
\Abs{\frac{1}{2\pi i}\int_{\gamma_{2}}f\ \d z} & \leq\epsilon\max_{z\in\gamma_{2}}\Abs{z^{(a+1)n+\alpha+1}\frac{(1-\lambda z)^{n+\beta}}{(z-\lambda)^{n+1}}}\\
 & \leq\epsilon^{(a+1)n+\alpha+2}\max_{z\in\gamma_{2}}\Abs{\frac{(1-\lambda z)^{n+\beta}}{(z-\lambda)^{n+1}}},
\end{align*}
which goes to $0$ as $\epsilon\rightarrow0$.
\item We have $z^{y}=e^{y\Log(z)}=e^{y\abs{z}+iy\Arg(z)}$ with the main
branch of the argument and so 
\begin{align*}
\int_{\gamma_{3}}f\ \d z & =\int_{-x}^{0}z^{(a+1)n+\alpha}\frac{(1-\lambda z)^{\beta}}{z-\lambda}\left(\frac{1-\lambda z}{z-\lambda}\right)^{n}\Bigg|_{z=t+i\tilde{\varepsilon}}\d t\\
 & \xrightarrow{\tilde{\varepsilon}\rightarrow0}\int_{-x}^{0}\abs{t}^{(a+1)n+\alpha}e^{((a+1)n+\alpha)\pi i}\frac{\abs{1-\lambda t}^{\beta}}{\abs{t}e^{\pi i}-\lambda}\left(\frac{1-\lambda\abs{t}e^{\pi i}}{\abs{t}e^{\pi i}-\lambda}\right)^{n}\d t\\
 & =e^{((a+1)n+\alpha)\pi i}(-1)^{n+1}\int_{-x}^{0}\abs{t}^{(a+1)n+\alpha}\frac{\abs{1-\lambda t}^{\beta}}{\abs{t}+\lambda}\left(\frac{1+\lambda\abs{t}}{\abs{t}+\lambda}\right)^{n}\d t\\
 & =e^{((a+1)n+\alpha)\pi i}(-1)^{n+1}\int_{0}^{x}t^{(a+1)n+\alpha}\frac{(1+\lambda t)^{\beta}}{t+\lambda}\left(\frac{1+\lambda t}{t+\lambda}\right)^{n}\d t.
\end{align*}
Similarly we find that 
\begin{align*}
 & \int_{\gamma_{4}}f\ \d z=\\
 & -\int_{0}^{{x}}z^{(a+1)n+\alpha}\frac{(1-\lambda z)^{\beta}}{z-\lambda}\left(\frac{1-\lambda z}{z-\lambda}\right)^{n}\Bigg|_{z=-t-i\tilde{\varepsilon}}\d t\\
 & \xrightarrow{\tilde{\varepsilon}\rightarrow0}-\int_{0}^{{x}}\abs{t}^{(a+1)n+\alpha}e^{-((a+1)n+\alpha)\pi i}\frac{\abs{1+\lambda t}^{\beta}}{\abs{t}e^{-\pi i}-\lambda}\left(\frac{1-\lambda\abs{t}e^{-\pi i}}{\abs{t}e^{-\pi i}-\lambda}\right)^{n}\d t\\
 & =-e^{-((a+1)n+\alpha)\pi i}(-1)^{n+1}\int_{0}^{{x}}t^{(a+1)n+\alpha}\frac{(1+\lambda t)^{\beta}}{t+\lambda}\left(\frac{1+\lambda t}{t+\lambda}\right)^{n}\d t.
\end{align*}
Summing the two terms shows that 
\begin{align*}
 & \frac{1}{2\pi i}\int_{\gamma_{3}\oplus\gamma_{4}}f\ \d z\ {\xrightarrow{\tilde{\varepsilon}\rightarrow0}}\\
 & (-1)^{n+1}\frac{\sin{(((a+1)n+\alpha)\pi})}{\pi}\int_{0}^{{x}}t^{(a+1)n+\alpha}\frac{(1+\lambda t)^{\beta}}{t+\lambda}\left(\frac{1+\lambda t}{t+\lambda}\right)^{n}\d t.
\end{align*}
\item To evaluate the integral over $\gamma_{5}$ we note that $f$ is holomorphic
in a dotted neighborhood around $\lambda$. We compute 
\begin{align*}
\frac{1}{2\pi i}\int_{\gamma_{5}}f\ \d z & =\frac{s^{-n}}{2\pi}\int_{-\pi}^{\pi}(\lambda+se^{i\phi})^{(a+1)n+\alpha}(1-\lambda^{2}-s\lambda e^{i\varphi})^{n+\beta}e^{-in\varphi}\d\varphi.
\end{align*}
For $s\in(0,\,\lambda)$ 
\[
(\lambda+se^{i\phi})^{(a+1)n+\alpha}=\sum_{\mu=0}^{\infty}\binom{(a+1)n+\alpha}{\mu}\lambda^{(a+1)n+\alpha-\mu}(se^{i\phi})^{\mu},
\]
and for $s\in(0,\,\frac{1-\lambda^{2}}{\lambda})$ 
\[
(1-\lambda^{2}-s\lambda e^{i\phi})^{n+\beta}=\sum_{\nu=0}^{\infty}\binom{{n+\beta}}{\nu}(1-\lambda^{2})^{{n+\beta}-\nu}(-\lambda se^{i\phi})^{\nu}
\]
and we conclude 
\begin{align*}
 & \frac{1}{2\pi i}\int_{\gamma_{5}}f\ \d z=\\
 & \sum_{\mu,\nu=0}^{\infty}\binom{(a+1)n+\alpha}{\mu}\binom{{n+\beta}}{\nu}\lambda^{(a+1)n+\alpha-\mu}(-\lambda)^{\nu}(1-\lambda^{2})^{{n+\beta}-\nu}s{}^{\mu+\nu}\\
 & \cdot\frac{s^{-n}}{2\pi}\int_{-\pi}^{\pi}e^{-i(\nu+\mu-n)\phi}\frac{{\rm d}\phi}{2\pi}.
\end{align*}
The last integral is $0$ unless $\nu+\mu-n=0$ so that we set $\nu=n-\mu$
and find 
\begin{align*}
 & \frac{1}{2\pi i}\int_{\gamma_{5}}f\ \d z\\
 & =\sum_{\mu=0}^{n}\binom{(a+1)n+\alpha}{\mu}\binom{{n+\beta}}{n-\mu}\lambda^{(a+1)n+\alpha-\mu}(-\lambda)^{n-\mu}(1-\lambda^{2})^{\beta+\mu}\\
 & =\lambda^{an+\alpha}(1-\lambda^{2})^{\beta}\sum_{\mu=0}^{n}\binom{(a+1)n+\alpha}{\mu}\binom{{n+\beta}}{n-\mu}(-\lambda^{2})^{n-\mu}(1-\lambda^{2})^{\mu}\\
 & =\lambda^{an+\alpha}(1-\lambda^{2})^{\beta}P_{n}^{(an+\alpha,\,\beta)}(1-2\lambda^{2}).
\end{align*}
\end{itemize}
\end{proof}

\section{Asymptotic analysis of integral representation}

\label{asymptInt} 

\subsection{Generalized Fourier integral via the method of stationary phase }

\label{sub:The-stationary-phase}

This section is devoted to the study of the large $n$ behavior of
the integral 
\[
\int_{0}^{\pi}z^{\alpha+1}\frac{(1-\lambda z)^{\beta}}{z-\lambda}\bigg(z^{a+1}\frac{1-\lambda z}{z-\lambda}\bigg)^{n}\Bigg|_{z=e^{i\varphi}}\d\varphi
\]
when $a\in[-\frac{2\lambda}{1+\lambda},\frac{2\lambda}{1+\lambda}]$. This is achieved using standard methods~\cite{AE},~\cite[Chapter~6.5]{BO}
from the theory of generalized Fourier-type integrals. The main idea
behind the~method of stationary phase is that the
leading order contribution to the integral comes from a small neighborhood
around the stationary points of $h$. To apply the method we observe
that if $z=e^{i\varphi}$ then $\Abs{z^{a+1}\frac{1-\lambda z}{z-\lambda}}=1$
and we introduce the real functions $g,h$ by the relations
\begin{align*}
\varphi & \mapsto g(\varphi)=z^{\alpha+1}\frac{(1-\lambda z)^{\beta}}{z-\lambda}\Bigg|_{z=e^{i\varphi}},\\
\varphi & \mapsto h(\varphi)={-i}\Log\left(\frac{z^{a+1}(1-\lambda z)}{z-\lambda}\right)\Bigg|_{z=e^{i\varphi}}=\Arg\left(\frac{z^{a+1}(1-\lambda z)}{z-\lambda}\right)\Bigg|_{z=e^{i\varphi}}.
\end{align*}
With this notation we can write out the integral in Fourier form 
\begin{align*}
\int_{0}^{\pi}z^{\alpha+1}\frac{(1-\lambda z)^{\beta}}{z-\lambda}\bigg(z^{a+1}\frac{1-\lambda z}{z-\lambda}\bigg)^{n}\Bigg|_{z=e^{i\varphi}}\d\varphi=\int_{0}^{\pi}g(\varphi)e^{inh(\varphi)}\d\varphi.
\end{align*}
The first step is to gain a picture about the stationary points of
$h$ on $[0,\pi]$. Computing derivatives by the chain rule yields
\begin{align*}
i\frac{dh}{dz} & =-\frac{1}{z-\lambda}+\frac{a+1}{z}-\frac{\lambda}{1-\lambda z},\\
i\frac{d^2h}{dz^2} & =\frac{1}{(z-\lambda)^{2}}-\frac{a+1}{z^{2}}-\frac{\lambda^{2}}{(1-\lambda z)^{2}},\\
i\frac{d^{3}h}{dz^{3}} & =-\frac{2}{(z-\lambda)^{3}}+\frac{2(a+1)}{z^{3}}-\frac{2\lambda^{3}}{(1-\lambda z)^{3}}.
\end{align*}
The function $h(\varphi)$ has a stationary point if and only if $dh/dz=0$,
i.e. iff 
\begin{align*}
a=\frac{\lambda}{z-\lambda}+\frac{\lambda z}{1-\lambda z}.
\end{align*}
Solving the latter for $z$ gives 
\[
z_{+,-}=\frac{a+a\lambda^{2}+2\lambda^{2}}{2\lambda(a+1)}\pm i\sqrt{1-\left(\frac{a+a\lambda^{2}+2\lambda^{2}}{2\lambda(a+1)}\right)^{2}}\in\partial\mathbb{D}
\]
and we write $z_{+,-}=e^{i\varphi_{+,-}}$ with $\varphi_{+}\in[0,\pi]$
and $\varphi_{-}\in(-\pi,0]$. For the second derivative we have that
\[
h''(\varphi)=\frac{d}{d\varphi}\left(\frac{dh}{dz}\frac{dz}{d\varphi}\right)=\frac{d^{2}h}{dz^{2}}\left(\frac{dz}{d\varphi}\right)^{2}+\frac{dh}{dz}\frac{d^{2}z}{(d\varphi)^{2}}.
\]
We distinguish the two cases \emph{1)} $a\in(-\frac{2\lambda}{1+\lambda},\,\frac{2\lambda}{1-\lambda})$
and \emph{2)} $a\in\left\{ -\frac{2\lambda}{1+\lambda},\,\frac{2\lambda}{1-\lambda}\right\} $,
which are characterized by the presence of a stationary point of order
one ($h'(\varphi_{+})=0$ but $h''(\varphi_{+})\neq0$) in \emph{Case
1)} and of order two ($h'(\varphi_{+})=h''(\varphi_{+})=0$ but $h'''(\varphi_{+})\neq0$)
in \emph{Case 2)}. 
We are ready to present the asymptotic behavior of the Fourier-type integral.  
\begin{prop}
\label{thm:delta_negative} Let $\lambda\in(0,\,1),$ $a\in[-\frac{2\lambda}{1+\lambda},\,\frac{2\lambda}{1-\lambda}]$
and $\alpha,\,\beta\in\mathbb{R}$. We have the following asymptotic
expansion as $n\rightarrow\infty$: 
\begin{enumerate}
\item If $a\in(-\frac{2\lambda}{1+\lambda},\,\frac{2\lambda}{1-\lambda})$
then 
\begin{align*}
 & \frac{1}{\pi}\Re{\left\{ \int_{0}^{\pi}z^{\alpha+1}\frac{(1-\lambda z)^{\beta}}{z-\lambda}\bigg(z^{a+1}\frac{1-\lambda z}{z-\lambda}\bigg)^{n}\Bigg|_{z=e^{i\varphi}}\d\varphi\right\} }\\
 & {=}\sqrt{\frac{2}{n\pi}}{\frac{(1-\lambda^{2})^{\frac{\beta}{2}}}{(a+1)^{\frac{\beta}{2}}}}\frac{\cos{\left((n+1)h(\varphi_{+})+(\alpha-a)\varphi_{+}+(\beta-1)\psi{+}\frac{\pi}{4}\right)}}{\left[{(1-\lambda^{2})}((a+2)\lambda+a)((a+2)\lambda-a)\right]^{\frac{1}{4}}}\\
 & +\frac{\sin\left((an+\alpha)\pi\right)}{n\pi}\frac{(1+\lambda)^{\beta}}{a(1+\lambda)+2\lambda}+\cO\left(\frac{1}{n^{3/2}}\right),
\end{align*}
where the parameters $\varphi_{+},\psi\in[0,\pi]$ are defined via the relations 
\begin{align*}
e^{i\varphi_{+}}=\frac{a+a\lambda^{2}+2\lambda^{2}}{2\lambda(a+1)}+i\sqrt{1-\left(\frac{a+a\lambda^{2}+2\lambda^{2}}{2\lambda(a+1)}\right)^{2}}
\end{align*}
and 
\begin{align*}
\sqrt{\frac{1-\lambda^{2}}{a+1}}e^{i\psi}=1-\lambda z_{+}.
\end{align*}
\item If $a=\frac{2\lambda}{1-\lambda}$ then 
\begin{align*}
 & \frac{1}{\pi}\Re{\left\{ \int_{0}^{\pi}z^{\alpha+1}\frac{(1-\lambda z)^{\beta}}{z-\lambda}\bigg(z^{a+1}\frac{1-\lambda z}{z-\lambda}\bigg)^{n}\Bigg|_{z=e^{i\varphi}}\d\varphi\right\} }\\
 & =\frac{(1-\lambda)^{\beta}}{3^{2/3}\Gamma(2/3)n^{1/3}\lambda^{1/3}(1+\lambda)^{1/3}}\left(1+\cO\left(\frac{1}{n^{1/3}}\right)\right),
\end{align*}
and if $a=-\frac{2\lambda}{1+\lambda}$ then 
\begin{align*}
 & \frac{1}{\pi}\Re{\left\{ \int_{0}^{\pi}z^{\alpha+1}\frac{(1-\lambda z)^{\beta}}{z-\lambda}\bigg(z^{a+1}\frac{1-\lambda z}{z-\lambda}\bigg)^{n}\Bigg|_{z=e^{i\varphi}}\d\varphi\right\} }=\\
 & {\frac{1}{\pi}\cos{\left((an+\alpha+\frac{1}{6})\pi\right)}\frac{\Gamma(1/3)(1+\lambda)^{\beta}}{3^{2/3}}\left(\frac{1}{n\lambda(1-\lambda)}\right)^{1/3}} + \cO\left(\frac{1}{n^{2/3}}\right).
\end{align*}
\end{enumerate}
\end{prop}

\begin{proof}[Proof of Proposition~\ref{thm:delta_negative}]
\leavevmode
\begin{enumerate}
\item If $a\in(-\frac{2\lambda}{1+\lambda},\,\frac{2\lambda}{1-\lambda})$
then the zeros $z_{+}=e^{i\varphi_{+}}$ and $z_{-}=e^{i\varphi_{-}}$
of $h'$ are distinct points located on $\partial\mathbb{D}$ with
$\varphi_{+}\in[0,\pi]$ and $\varphi_{-}\in(-\pi,0]$. Only $z_{+}$ is interesting because
we integrate over $[0,\pi]$. Plugging in
we see that
\begin{eqnarray*}
i\frac{d^{2}h}{dz^{2}}\Bigg|_{z=z_{+}} & = & \frac{(1-\lambda^{2})(1-z_{+}^{2})\lambda}{z_{+}(z_{+}-\lambda)^{2}(1-\lambda z_{+})^{2}}
\end{eqnarray*}
so that $h''(\varphi_{+})>0$. To find the asymptotics we apply a
standard result by A. Erdélyi~\cite[Theorem 4]{AE} (see also F.
Olver \cite[Theorem 1]{FO1} for a more explicit form), which however
requires that the stationary point is an endpoint of the interval
of integration. Hence we begin by splitting 
\[
\int_{0}^{\pi}g(\varphi)e^{inh(\varphi)}\d\varphi=\int_{0}^{\varphi_{+}}g(\varphi)e^{inh(\varphi)}\d\varphi+\int_{\varphi_{+}}^{\pi}g(\varphi)e^{inh(\varphi)}\d\varphi.
\]
For the second integral \cite[Theorem 4]{AE}
gives 
\begin{align*}
\int_{\varphi_{+}}^{\pi}g(\varphi)e^{inh(\varphi)}\d\varphi & =\frac{1}{2}\Gamma(1/2)k(0)e^{i\frac{\pi}{4}}n^{-1/2}e^{inh(\varphi_{+})}\\
 & +\frac{1}{2}\Gamma(1)k'(0)e^{i\frac{\pi}{2}}n^{-1}e^{inh(\varphi_{+})}\\
 & -\frac{i}{n}e^{inh(\pi)}\frac{g(\pi)}{h'(\pi)}+\cO\left(\frac{1}{n^{3/2}}\right),
\end{align*}
with 
\[
k(0)=2^{1/2}g(\varphi_{+})\left(h''(\varphi_{+})\right)^{-1/2}
\]
and 
\[
k'(0)=\frac{2}{h''(\varphi_{+})}g'(\varphi_{+})-\frac{2}{h''(\varphi_{+})}\frac{h^{(3)}(\varphi_{+})}{3h''(\varphi_{+})}g(\varphi_{+}).
\]
For the first integral $\int_{0}^{\varphi_{+}}g(\varphi)e^{inh(\varphi)}\d\varphi$
we change the variable of integration $\varphi\mapsto-\varphi$ as
suggested in \cite[page 23]{AE}. We get 
\[
\int_{0}^{\varphi_{+}}g(\varphi)e^{inh(\varphi)}\d\varphi=\int_{-\varphi_{+}}^{0}g(-\varphi)e^{inh(-\varphi)}\d\varphi.
\]
Applying \cite[Theorem 4]{AE} (see also \cite[Theorem 1]{FO1}) gives
\begin{align*}
\int_{-\varphi_{+}}^{0}g(-\varphi)e^{inh(-\varphi)}\d\varphi & =\frac{1}{2}\Gamma(1/2)k(0)e^{i\frac{\pi}{4}}n^{-1/2}e^{inh(\varphi_{+})} +\frac{1}{2}\Gamma(1)k'(0)e^{i\frac{\pi}{2}}n^{-1}e^{inh(\varphi_{+})}\\
 & -\frac{i}{n}e^{inh(0)}\frac{g(0)}{h'(0)}+\cO\left(\frac{1}{n^{3/2}}\right)
\end{align*}
with 
\[
k(0)=2^{1/2}g(\varphi_{+})\left(h''(\varphi_{+})\right)^{-1/2}
\]
and 
\[
k'(0)=-\frac{2}{h''(\varphi_{+})}g'(\varphi_{+})+\frac{2}{h''(\varphi_{+})}\frac{h^{(3)}(\varphi_{+})}{3h''(\varphi_{+})}g(\varphi_{+}).
\]
Observing that $h(0)=0$ and $h(\pi)=a\pi,$ $g(0)=(1-\lambda)^{\beta-1}$
while $g(\pi)=e^{i(\alpha+1)\pi}(1+\lambda)^{\beta-1},$ and $h'(0)=\frac{a(1-\lambda)-2\lambda}{1-\lambda}$
while $h'(\pi)=-\frac{a(1+\lambda)+2\lambda}{1+\lambda}$ we get 
\[
-\frac{i}{n}e^{inh(\pi)}\frac{g(\pi)}{h'(\pi)}=\frac{i}{n}e^{i(an+\alpha)\pi}\frac{(1+\lambda)^{\beta}}{a(1+\lambda)+2\lambda}
\]
and
\[
-\frac{i}{n}e^{inh(0)}\frac{g(0)}{h'(0)}=-\frac{i}{n}\frac{(1-\lambda)^{\beta}}{a(1-\lambda)-2\lambda}.
\]
We conclude that 
\begin{align*}
&\int_{0}^{\pi}g(\varphi)e^{inh(\varphi)}\d\varphi\\
& =\Gamma(1/2)\left(2^{1/2}g(\varphi_{+})\left(h''(\varphi_{+})\right)^{-1/2}\right)e^{i\frac{\pi}{4}}n^{-1/2}e^{inh(\varphi_{+})}\\
 & +\frac{i}{n}e^{i(an+\alpha)\pi}\frac{(1+\lambda)^{\beta}}{a(1+\lambda)+2\lambda}-\frac{i}{n}\frac{(1-\lambda)^{\beta}}{a(1-\lambda)-2\lambda}+\cO\left(\frac{1}{n^{3/2}}\right)\\
 & =e^{inh(\varphi_{+}){+}i\frac{\pi}{4}}z_{+}^{\alpha+1}\frac{(1-\lambda z_{+})^{\beta}}{z_{+}-\lambda}\left(\frac{2\abs{z_{+}-\lambda}^{4}}{n\lambda(1-\lambda^{2})\abs{1-z_{+}^{2}}}\right)^{1/2}\Gamma(1/2)\\
 & +\frac{1}{n}e^{i\left((an+\alpha)\pi+\frac{\pi}{2}\right)}\frac{(1+\lambda)^{\beta}}{a(1+\lambda)+2\lambda}-\frac{i}{n}\frac{(1-\lambda)^{\beta}}{a(1-\lambda)-2\lambda}+\cO\left(\frac{1}{n^{3/2}}\right).
\end{align*}
We set $\psi=\Arg{(1-\lambda z_{+})}$ and obtain 
\begin{align*}
 & e^{inh(\varphi_{+}){+}i\frac{\pi}{4}}z_{+}^{\alpha+1}\frac{(1-\lambda z_{+})^{\beta}}{z_{+}-\lambda}\left(\frac{2\abs{z_{+}-\lambda}^{4}}{n\lambda(1-\lambda^{2})\abs{1-z_{+}^{2}}}\right)^{1/2}\Gamma(1/2)\\
 & =\frac{\sqrt{2\pi}\abs{1-\lambda z_{+}}^{\beta-1+2}}{\sqrt{n\lambda(1-\lambda^{2})\abs{1-z_{+}^{2}}}}e^{i\left(nh(\varphi_{+}){+}\frac{\pi}{4}{+(\alpha-a)\varphi_{+}+(\beta-1)\psi+h(\varphi_{+})}\right)}\\
 & =\sqrt{\frac{2\pi}{n}}\frac{(1-\lambda^{2})^{\frac{\beta}{2}-\frac{1}{4}}}{(a+1)^{\frac{\beta}{2}}}\frac{e^{i\left((n+1)h(\varphi_{+})+(\alpha-a)\varphi_{+}+(\beta-1)\psi{+}\frac{\pi}{4}\right)}}{\left(((a+2)\lambda+a)((a+2)\lambda-a)\right){}^{\frac{1}{4}}},
\end{align*}
where we made use of 
\[
\abs{1-\lambda z_{+}}={\abs{z_{+}-\lambda}}=\sqrt{\frac{1-\lambda^{2}}{a+1}}
\]
and 
\begin{align*}
&\abs{z_{+}^{2}-1}=2\sqrt{1-\left(\frac{a+a\lambda^{2}+2\lambda^{2}}{2\lambda(a+1)}\right)^{2}}\\
&=\frac{\sqrt{(1-\lambda^{2})((a+2)\lambda+a)((a+2)\lambda-a)}}{\lambda(a+1)}.
\end{align*}
Taking the real part of $\int_{0}^{\pi}g(\varphi)e^{inh(\varphi)}\d\varphi$
completes the proof of Theorem \ref{thm:delta_negative}, point 1).\\

\item If $a\in\{-\frac{2\lambda}{1+\lambda},\,\frac{2\lambda}{1-\lambda}\}$
then $h'$ has a unique zero. If $a=\frac{2\lambda}{1-\lambda}$ then
$z_{+}=1$ and 
\[
h(0)=h'(0)=h''(0)=0,
\]
while 
\[
h^{(3)}(0)=\frac{2\lambda(1+\lambda)}{(1-\lambda)^{3}}.
\]
Applying~ \cite[Theorem 4]{AE} at the second order in this case
yields the asymptotic behavior 
\[
\int_{0}^{\pi}g(\varphi)e^{inh(\varphi)}\d\varphi=g(\varphi_{+})e^{inh(\varphi_{+}){+}i\frac{\pi}{6}}\left(\frac{6}{n\abs{h^{(3)}(\varphi_{+})}}\right)^{1/3}\frac{\Gamma(1/3)}{3}+\cO\left(\frac{1}{n^{2/3}}\right).
\]
Direct computation shows that 
\begin{align*}
&\int_{0}^{\pi}g(\varphi)e^{inh(\varphi)}\d\varphi\\ & =\frac{\Gamma(1/3)}{3}\frac{z_{+}^{\alpha+1}(1-\lambda z_{+})^{\beta}e^{i\frac{\pi}{6}}}{z_{+}-\lambda}\left(\frac{z_{+}^{a+1}(1-\lambda z_{+})}{z_{+}-\lambda}\right)^{n}\left(\frac{3(1-\lambda)^{3}}{n\lambda(1+\lambda)}\right)^{1/3}\left(1+\cO\left(\frac{1}{n^{1/3}}\right)\right)\\
 & =e^{i\frac{\pi}{6}}(1-\lambda)^{\beta}\left(\frac{3}{n\lambda(1+\lambda)}\right)^{1/3}\left(1+\cO\left(\frac{1}{n^{1/3}}\right)\right),
\end{align*}
and we conclude
\begin{align*}
&\frac{1}{\pi}\Re{\left\{ \int_{0}^{\pi}z^{\alpha+1}\frac{(1-\lambda z)^{\beta}}{z-\lambda}\bigg(z^{a+1}\frac{1-\lambda z}{z-\lambda}\bigg)^{n}\Bigg|_{z=e^{i\varphi}}\d\varphi\right\} }\\
&=\frac{(1-\lambda)^{\beta}}{3^{2/3}\Gamma(2/3)}\left(\frac{1}{n\lambda(1+\lambda)}\right)^{1/3}\left(1+\cO\left(\frac{1}{n^{1/3}}\right)\right).
\end{align*}
If $a=-\frac{2\lambda}{1+\lambda}$ then $z_{+}=-1$.
Changing the variable of integration $\varphi\mapsto-\varphi$ and making use of $\overline{g(\varphi)}=g(-\varphi)$, $\overline{e^{inh(\varphi)}}=e^{inh(-\varphi)}$ one verifies that 
\[
\int_{-\pi}^{0}g(\varphi)e^{inh(\varphi)}\d\varphi=\overline{\int_{0}^{\pi}g(\varphi)e^{inh(\varphi)}\d\varphi}
\]
so that the new saddle point $-\pi$ is the left endpoint of the interval of integration; indeed we have 
\[
h'(-\pi)=h''(-\pi)=0,\qquad h^{(3)}(-\pi)=-\frac{2\lambda(1-\lambda)}{(1+\lambda)^{3}}.
\]
Applying again~\cite[Theorem 4]{AE} yields the asymptotic
behavior 
\begin{align*}
\int_{-\pi}^{0}g(\varphi)e^{inh(\varphi)}\d\varphi=g(-\pi)e^{inh(-\pi)-i\frac{\pi}{6}}\left(\frac{6}{n\abs{h^{(3)}(-\pi)}}\right)^{1/3}\frac{\Gamma(1/3)}{3}+ \cO\left(\frac{1}{n^{2/3}}\right).
\end{align*}
With
\begin{align*}
 & g(-\pi)=e^{-i\pi(\alpha+1)}\frac{(1+\lambda)^{\beta}}{-(1+\lambda)}=(1+\lambda)^{\beta-1}e^{-i\alpha\pi},\\
 & h(-\pi)=\textnormal{Arg}\left(\frac{e^{-i\pi(a+1)}(1+\lambda)}{-(1+\lambda)}\right)=-a\pi
\end{align*}
we conclude that
\begin{align*}
 & \frac{1}{\pi}\Re{\left\{ \int_{0}^{\pi}z^{\alpha+1}\frac{(1-\lambda z)^{\beta}}{z-\lambda}\bigg(z^{a+1}\frac{1-\lambda z}{z-\lambda}\bigg)^{n}\Bigg|_{z=e^{i\varphi}}\d\varphi\right\} }\\
 & =\frac{1}{\pi}\cos{\left(\left(an+\alpha+\frac{1}{6}\right)\pi\right)}\frac{\Gamma(1/3)(1+\lambda)^{\beta}}{3^{2/3}}\left(\frac{1}{n\lambda(1-\lambda)}\right)^{1/3}+ \cO\left(\frac{1}{n^{2/3}}\right).
\end{align*}
\end{enumerate}
\end{proof}
\subsection{The Laplace-type integral via Laplace's method}

\label{sub:Laplace's-method-for} This section is devoted to the study
of the large $n$ behavior of the integral 
\[
\int_{0}^{1}\frac{(1+\lambda t)^{\beta}t^{\alpha}}{t+\lambda}\left(t^{(a+1)}\frac{1+\lambda t}{t+\lambda}\right)^{n}{\rm d}t.
\]
This is achieved using standard methods~\cite{BO,FO0} from the theory of Laplace-type integrals. The main idea behind Laplace's method is that if the real continuous function $h$ has a maximum at
a point $\xi\in[0,1]$ and if $f(\xi)\neq0$ then as $n$ grows large
only values in an immediate neighborhood of $\xi$ contribute
to the integral, see~\cite{BO,FO0}. To apply the method we introduce
for $\alpha,\beta\in\mathbb{R}$, $a>-1$ and $\lambda\in(0,1)$ (overwriting
$f,g,h$ of Section~\ref{sub:The-stationary-phase}) the real functions
\begin{align*}
t\mapsto f(t) & =\frac{(1+\lambda t)^{\beta}t^{\alpha}}{t+\lambda},\ t\in[0,1],\\
t\mapsto g(t) & =t^{a+1}\frac{1+\lambda t}{\lambda+t},\ t\in[0,1],\\
t\mapsto h(t) & =\ln{(g(t))},\ t\in(0,1].
\end{align*}
With this notation we can write out the integral in Laplace form 
\[
\int_{0}^{1}\frac{(1+\lambda t)^{\beta}t^{\alpha}}{t+\lambda}\left(t^{(a+1)}\frac{1+\lambda t}{t+\lambda}\right)^{n}{\rm d}t=\int_{0}^{1}f(t)e^{nh(t)}{\rm d}t.
\]
The first step is to gain a picture about the behavior of $h$ on
$(0,1]$. Computing derivatives we find 
\begin{align*}
h'(t) & =\frac{a+1}{t}-\frac{1}{t+\lambda}+\frac{\lambda}{1+\lambda t},\\
h''(t) & =-\frac{a+1}{t^{2}}+\frac{1}{(t+\lambda)^{2}}-\frac{\lambda^{2}}{(1+\lambda t)^{2}},\\
h'''(t) & =\frac{2a+2}{t^{3}}-\frac{2}{(t+\lambda)^{3}}+\frac{2\lambda^{3}}{(1+\lambda t)^{3}}.
\end{align*}
We note that $h'(t)=0$ holds iff 
\begin{align*}
a=\frac{t}{\lambda+t}-\frac{\lambda t}{1+\lambda t}-1.
\end{align*}
For $t$ this yields two solutions 
\begin{equation}
t_{\pm}=-\frac{a+a\lambda^{2}+2\lambda^{2}}{2\lambda(a+1)}\pm\sqrt{\left(\frac{a+a\lambda^{2}+2\lambda^{2}}{2\lambda(a+1)}\right)^{2}-1},\label{eq:tplusminus}
\end{equation}
which are real iff $(a\geq\frac{2\lambda}{1-\lambda}$ or $a\leq\frac{-2\lambda}{1+\lambda})$. 
\begin{lem}
\label{trivial} Let $h(t)=\ln{\left(t^{a+1}\frac{1+t\lambda}{t+\lambda}\right)}$,
$t\in(0,1]$, $\lambda\in(0,1)$ and let $a>-1$. For $h$ we have
that $\lim_{t\rightarrow0^{+}}h(t)=-\infty$ and $h(1)=0$ and the
following properties:
\begin{enumerate}
\item If $a>\frac{-2\lambda}{1+\lambda}$ then $h(t)\leq0$ and $h'(t)>0$
for $t\in(0,\,1]$.
\item If $a=\frac{-2\lambda}{1+\lambda}$ then $t_{\pm}=1$ and $h(t)<0$
and $h'(t)>0$ for $t\in(0,\,1)$. Further $h'(1)=h''(1)=0$ and $h^{(3)}(1)=\frac{2\lambda(1-\lambda)}{(1+\lambda)^{3}}$.
\item If $a\in(-1,\,\frac{-2\lambda}{1+\lambda})$ then $h$ has a unique
maximum at $t_{-}\in(0,\,1)$ with 
\begin{align*}
 & h(t_{-})>0,\\
 & h'(t_{-})=0,\\
 & h''(t_{-})=\frac{\lambda}{t_{-}}\left(-\frac{1}{(\lambda+t_{-})^{2}}+\frac{1}{(1+\lambda t_{-})^{2}}\right)<0.
\end{align*}
\end{enumerate}
\end{lem}

For illustration Figure~\ref{IHaPl} depicts the three cases of the lemma.
\begin{figure}[h]
\centering
\begin{subfigure}{.3\textwidth}
\centering
\includegraphics[width=\linewidth]{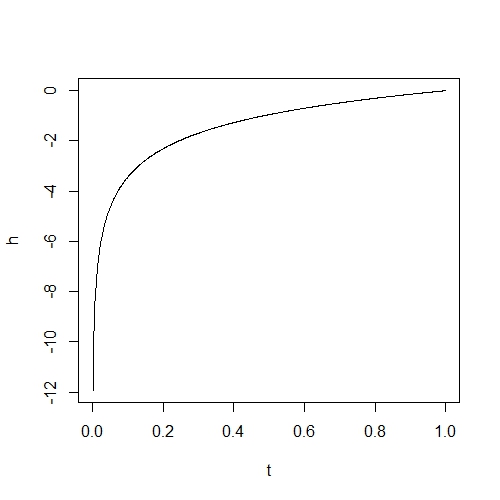}
\caption{}
\end{subfigure}
\begin{subfigure}{.3\textwidth}
\centering
\includegraphics[width=\linewidth]{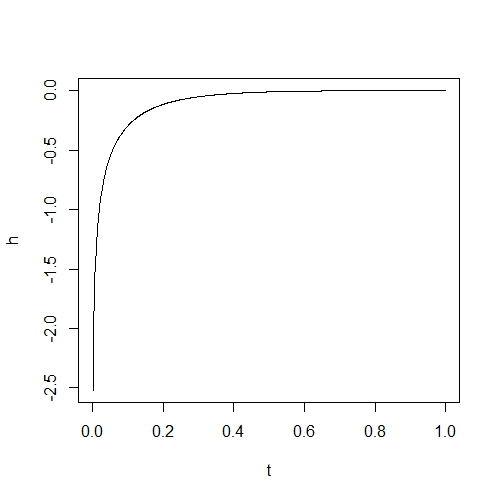}
\caption{}
\end{subfigure}
\centering
\begin{subfigure}{.3\textwidth}
\centering
\includegraphics[width=\linewidth]{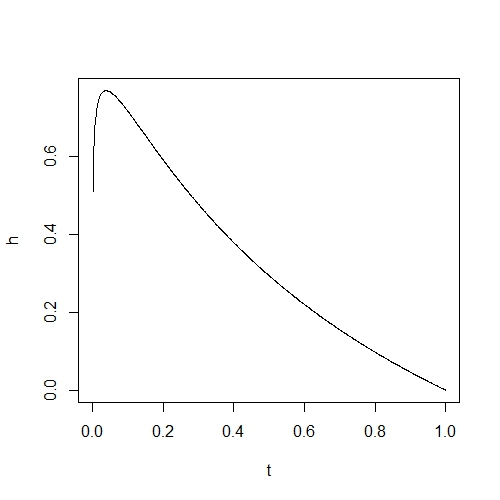}
\caption{}
\end{subfigure}
\caption{\emph{Graphs of $h$ in different parameter regions.} In all plots $\lambda=0.3$ is chosen. A) corresponds to point (1) in the lemma and the graph illustrates $a=0.9>-6/13$. B) corresponds to point (2) in the lemma and the graph illustrates $a=-6/13$. C) corresponds to point (3) in the lemma and the graph illustrates $a=-0.9<-6/13$.}
\label{IHaPl}
\end{figure}

\begin{proof}
The proof is a consequence of the from of the derivatives of $h$. 
\begin{enumerate}
\item The function $\frac{t}{\lambda+t}-\frac{\lambda t}{1+\lambda t}-1$
is increasing on $[0,1]$ and achieves the maximum $-\frac{2\lambda}{1+\lambda}$
at $t=1$. Hence, if $a>\frac{-2\lambda}{1+\lambda}$ then $h'(t)=0$
has no solution on $(0,1]$ i.e.~$h'$ is either positive or negative
on the whole interval. Since $h(0)=-\infty$ and $h(1)=0$ it follows
that $h$ is strictly increasing on $[0,1)$ and subsequently we also
have that $h(t)\leq0$ on $(0,1]$. 
\item It holds that $t_{\pm}\left(a=\frac{-2\lambda}{1+\lambda}\right)=1$
and that $h'(1)=h''(1)=0$ and $h'''(1)=\frac{2\lambda(1-\lambda)}{(1+\lambda)^{3}}$.
As in (1) we can conclude that $t=1$ is the only zero of $h'$ in
$(0,1]$ so that $h'$ is strictly positive. Furthermore it holds
that $h(t)\leq0$ on $(0,1]$. 
\item If $a\in(-1,\,\frac{-2\lambda}{1+\lambda})$ then $t_{-}\in(0,1)$
and $t_{+}>1$ so that $h'$ has a unique zero $t_{-}\in(0,1)$. Plugging
this into our formula for $h''(t)$ we find that 
\begin{align*}
h''(t_{-})=\frac{\lambda}{t_{-}}\left(-\frac{1}{(\lambda+t_{-})^{2}}+\frac{1}{(1+\lambda t_{-})^{2}}\right)<0.
\end{align*}
We conclude that $t_{-}$ is a local maximum of $h$ and that $h$
is decreasing in a neighborhood of $t_{-}$. Since $\lim_{t\rightarrow0+}h(t)=-\infty$
and $h(1)=0$ and there are no other critical points of $h$ we conclude
that $h(t_{-})$ is a global maximum of $h$ and therefore it is positive. 
\end{enumerate}
\end{proof}
We are ready to present the asymptotic behavior of the Laplace-type
integral. 
\begin{prop}
\label{Laplace} Let $\lambda\in(0,1)$, $a>-1$ and let $\alpha,\ \beta>-1$.
We have the following asymptotic behavior as $n\rightarrow\infty$: 
\begin{enumerate}
\item If $a>\frac{-2\lambda}{1+\lambda}$ then 
\begin{align*}
\int_{0}^{1}\frac{(1+\lambda t)^{\beta}t^{\alpha}}{t+\lambda}\left(t^{(a+1)}\frac{1+\lambda t}{t+\lambda}\right)^{n}{\rm d}t=\frac{1}{n}\frac{(1+\lambda)^{\beta}}{\lambda(a+2)+a}\left(1+\cO\left(\frac{1}{n}\right)\right).
\end{align*}
\item If $a=\frac{-2\lambda}{1+\lambda}$ then 
\begin{align*}
\int_{0}^{1}\frac{(1+\lambda t)^{\beta}t^{\alpha}}{t+\lambda}\left(t^{(a+1)}\frac{1+\lambda t}{t+\lambda}\right)^{n}{\rm d}t=\frac{\Gamma(1/3)}{3^{2/3}n^{1/3}}\frac{(1+\lambda)^{\beta}}{\left(\lambda(1-\lambda)\right)^{1/3}}\left(1+\cO\left(\frac{1}{n^{1/3}}\right)\right).
\end{align*}
\item If $a\in(-1,\,\frac{-2\lambda}{1+\lambda})$ then $g(t_{-})>1$ and
\begin{align*}
 & \int_{0}^{1}\frac{(1+\lambda t)^{\beta}t^{\alpha}}{t+\lambda}\left(t^{(a+1)}\frac{1+\lambda t}{t+\lambda}\right)^{n}{\rm d}t\\
 & =(1+\lambda t_{-})^{\beta+1}t_{-}^{\alpha}\left(g(t_{-})\right)^{n}\sqrt{\frac{2\pi t_{-}}{n\lambda\left((1+\lambda t_{-})^{2}-(\lambda+t_{-})^{2}\right)}}\left(1+\cO\left(\frac{1}{n^{1/2}}\right)\right),
\end{align*}
{where $t_{-}$ is defined in Equation~\eqref{eq:tplusminus}.}
\end{enumerate}
\end{prop}

\begin{proof}
For properties of $h$ we refer to the respective points of Lemma~\ref{trivial}.
We first rewrite our Laplace integral so that the assumptions in A. Erd\'elyi's asymptotic formula \cite[Theorem 7.1]{FO0} are met:
%
%
%We employ standard Laplace methods to study the asymptotic behavior of 
\begin{align*}
\int_{0}^{1}\frac{(1+\lambda t)^{\beta}t^{\alpha}}{t+\lambda}\left(t^{(a+1)}\frac{1+\lambda t}{t+\lambda}\right)^{n}{\rm d}t & =\int_{-1}^{0}\frac{(1-\lambda t)^{\beta}(-t)^{\alpha}}{\lambda-t}\left((-t)^{(a+1)}\frac{1-\lambda t}{\lambda-t}\right)^{n}{\rm d}t\\
 & =\int_{-1}^{0}\tilde{f}(t)e^{-n\tilde{h}(t)}{\rm d}t,
\end{align*}
where 
\begin{align*}
\tilde{f}(t) & =f(-t)=\frac{(1-\lambda t)^{\beta}(-t)^{\alpha}}{\lambda-t},\\
\tilde{h}(t) & =-h(-t)=-\ln\left((-t)^{a+1}\frac{1-\lambda t}{\lambda-t}\right).
\end{align*}
\begin{enumerate}
\item If $a>\frac{-2\lambda}{1+\lambda}$ then $\tilde{h}(-1)=0$, $\lim_{t\rightarrow0^{+}}\tilde{h}(t)=+\infty$
and for $t\in[-1,0)$ we have $\tilde{h}(t)\geq0$ and $\tilde{h}'(t)>0$.
Applying Erdélyi's result we get 
\begin{align*}
\int_{-1}^{0}\tilde{f}(t)e^{-n\tilde{h}(t)}{\rm d}t & =\frac{\tilde{f}(-1)}{n\tilde{h}'(-1)}e^{n\tilde{h}(-1)}\left(1+\cO\left(\frac{1}{n}\right)\right)\\
 & =\frac{1}{n}\frac{(1+\lambda)^{\beta}}{\lambda(a+2)+a}\left(1+\cO\left(\frac{1}{n}\right)\right)
\end{align*}
which completes the proof. 
\item If $a=-\frac{2\lambda}{1+\lambda}$ then $\tilde{h}$ is positive
and strictly increasing on $[-1,0)$ with $\tilde{h}(-1)=0$, $\lim_{t\rightarrow0^{-}}\tilde{h}(t)=\infty$.
Moreover $\tilde{h}'(-1)=\tilde{h}''(-1)=0$ while $\tilde{h}^{(3)}(-1)=h^{(3)}(1)=\frac{2\lambda(1-\lambda)}{(1+\lambda)^{3}}>0$.
Applying again Erdélyi's result we get 
\begin{align*}
\int_{-1}^{0}\tilde{f}(t)e^{-n\tilde{h}(t)}{\rm d}t & =\tilde{f}(-1)e^{-n\tilde{h}(-1)}\frac{\Gamma(1/3)}{3}\left(\frac{3!}{n\tilde{h}^{(3)}(-1)}\right)^{1/3}\left(1+\cO\left(\frac{1}{n^{1/3}}\right)\right)\\
 & =\frac{(1+\lambda)^{\beta-1}}{n^{1/3}}\frac{\Gamma(1/3)}{3^{2/3}}\frac{1+\lambda}{\left(\lambda(1-\lambda)\right)^{1/3}}\left(1+\cO\left(\frac{1}{n^{1/3}}\right)\right)\\
 & =\frac{1}{n^{1/3}}\frac{\Gamma(1/3)}{3^{2/3}}\frac{(1+\lambda)^{\beta}}{\left(\lambda(1-\lambda)\right)^{1/3}}\left(1+\cO\left(\frac{1}{n^{1/3}}\right)\right).
\end{align*}
\item If $a\in(-1,\,\frac{-2\lambda}{1+\lambda})$ then $t_{-}$ is the
unique maximum of $h$ on $(0,1)$ with $g(t_{-})=e^{h(t_{-})}>1$.
Applying a standard result \cite[Formula (6.4.19c)]{BO} we obtain
%Applying standard results about Laplace integrals \cite[formula (6.4.19c) page 267]{BO} 
\begin{align*}
&\int_{0}^{1}\frac{(1+\lambda t)^{\beta}t^{\alpha}}{t+\lambda}\left(t^{(a+1)}\frac{1+\lambda t}{t+\lambda}\right)^{n}{\rm d}t\\
&=\frac{(1+\lambda t_{-})^{\beta}t_{-}^{\alpha}}{t_{-}+\lambda}\sqrt{\frac{2\pi}{-nh''(t_{-})}}\left(g(t_{-})\right)^{n}\left(1+\cO\left(\frac{1}{n^{1/2}}\right)\right).
\end{align*}
Moreover 
\begin{align*}
\sqrt{\frac{2\pi}{-nh''(t_{-})}}=(\lambda+t_{-})(1+\lambda t_{-})\sqrt{\frac{2\pi t_{-}}{n\lambda\left((1+\lambda t_{-})^{2}-(\lambda+t_{-})^{2}\right)}}
\end{align*}
and the result follows. 
\end{enumerate}
\end{proof}

\section{Simple estimates in the exponential regions}

\label{sec:Simple-exponential-bounds} For applications it is often useful to identify parameter regions of exponential decay and to have simple upper estimates in such cases. The accumulated contribution of those regions is usually asymptotically
negligible with the dominant contribution arising from the asymptotically
larger parameter regions. The following proposition contains
a simple characteristics of parameter regions of exponential behavior.
\begin{prop}
\label{Second_thm} Let $\alpha,\ \beta>-1$, $a>-1$, $\lambda\in(0,\,1)$ and $n$ be a positive integer. 
\begin{enumerate}
\item We assume that $an+\alpha$ is an integer.
\begin{enumerate}
\item If $a>\frac{2\lambda}{1-\lambda}$ then there is $x^{*}\in(\lambda,1)$
such that $\left(x^{*}\right)^{a+1}\frac{1-\lambda x^{*}}{x^{*}-\lambda}\in(0,\,1)$
and 
\begin{align*}
 & \Abs{\lambda^{an+\alpha}(1-\lambda^{2})^{\beta}P_{n}^{(an+\alpha,\,\beta)}(1-2\lambda^{2})}\\
 & \leq\left(x^{*}\right)^{\alpha+1}\frac{(1+\lambda x^{*})^{\beta+1}}{(1-x^{*}\lambda)(x^{*}-\lambda)}\left(\left(x^{*}\right)^{a+1}\frac{1-\lambda x^{*}}{x^{*}-\lambda}\right)^{n}.
\end{align*}
$x^{*}$ is computed explicitly below.
\item If $-1<a<\frac{-2\lambda}{1+\lambda}$ then there is $x^{*}\in(1,1/\lambda)$
such that $\left(x^{*}\right)^{a+1}\frac{1+\lambda x^{*}}{x^{*}+\lambda}\in(0,1)$
and 
\begin{align*}
 & \Abs{\lambda^{an+\alpha}(1-\lambda^{2})^{bn+\beta}P_{n}^{(an+\alpha,\,\beta)}(1-2\lambda^{2})}\\
 & \leq\left(x^{*}\right)^{\alpha+1}\frac{(1+\lambda x^{*})^{\beta+1}}{(1-x^{*}\lambda)(x^{*}-\lambda)}\left(\left(x^{*}\right)^{a+1}\frac{1+\lambda x^{*}}{x^{*}+\lambda}\right)^{n}.
\end{align*}
$x^{*}$ is computed explicitly below.
\end{enumerate}
\item We assume that $an+\alpha$ is not an integer. 
\begin{enumerate}
\item If $a>\frac{2\lambda}{1-\lambda}$ then there is $x^{*}\in(0,1)$
such that $\left(x^{*}\right)^{a+1}\frac{1-\lambda x^{*}}{x^{*}-\lambda},\left(x^{*}\right)^{a+1}\frac{1+\lambda x^{*}}{x^{*}+\lambda}\in(0,1)$
and
\begin{align*}
 & \abs{\lambda^{an+\alpha}(1-\lambda^{2})^{\beta}P_{n}^{(an+\alpha,\,\beta)}(1-2\lambda^{2})}\\
 & \leq \left(x^{*}\right)^{\alpha+1}\frac{(1+\lambda x^{*})^{\beta+1}}{(1-x^{*}\lambda)(x^{*}-\lambda)}\\&\cdot\left[\left(\left(x^{*}\right)^{a+1}\frac{1-\lambda x^{*}}{x^{*}-\lambda}\right)^{n}+\left(\left(x^{*}\right)^{a+1}\frac{1+\lambda x^{*}}{x^{*}+\lambda}\right)^{n}\right].
\end{align*}
$x^{*}$ is computed explicitly below.
\item If $a\in(-1,-\frac{2\lambda}{1+\lambda})$ and if the sequence $(an+\alpha)_n$ is separate from integers then the quantity 
\[
\lambda^{an+\alpha}P_{n}^{(an+\alpha,\,\beta)}(1-2\lambda^{2})
\]
increases exponentially in magnitude. The precise behavior depends on the proximity of the sequence $(an+\alpha)_n$ to integers, see the proof of Theorem~\ref{gbar} point (4). 
\end{enumerate}
\end{enumerate}
\end{prop}

\begin{proof}[Proof of Proposition~\ref{Second_thm}, point 1]
Since $an+\alpha$ is assumed to be an integer it is sufficient to consider the Fourier integral. Clearly 
\begin{align*}
\Abs{\frac{1}{\pi}\Re{\left\{ \int_{0}^{\pi}z^{\alpha+1}\frac{(1-\lambda z)^{\beta}}{z-\lambda}\bigg(z^{a+1}\frac{1-\lambda z}{z-\lambda}\bigg)^{n}\Bigg|_{z=xe^{i\varphi}}\d\varphi\right\} }}\leq\max_{\Abs{z}=x}\left(\Abs{z^{\alpha+1}\frac{(1-\lambda z)^{\beta}}{z-\lambda}}\Abs{z^{a+1}\frac{1-\lambda z}{z-\lambda}}^{n}\right).
\end{align*}
Recall that for $z,w\in\mathbb{D}$ it holds that~\cite[Formula (1.11)]{JG}
\begin{align*}
\frac{\abs{z}-\abs{w}}{1-\abs{z}\abs{w}}\leq\Abs{\frac{z-w}{1-\bar{w}z}}\leq\frac{\abs{z}+\abs{w}}{1+\abs{z}\abs{w}}.
\end{align*}
This implies that 
\[
\max_{\abs{z}=x}\Abs{z^{a+1}\frac{1-\lambda z}{z-\lambda}}\leq\begin{cases}
x^{a+1}\frac{1-\lambda x}{x-\lambda},\ x\in(\lambda,1]\\
x^{a+1}\frac{1+\lambda x}{x+\lambda},\ x\in(1,\infty)
\end{cases}.
\]
\begin{enumerate}[(a)]
\item Suppose that $a>\frac{2\lambda}{1-\lambda}$. The derivative of $x\mapsto x^{a+1}\frac{1-\lambda x}{x-\lambda}$
has zeros at
\[
x_{\pm}=\frac{a+(a+2)\lambda^{2}}{2\lambda(a+1)}\pm\sqrt{\left(\frac{a+(a+2)\lambda^{2}}{2\lambda(a+1)}\right)^{2}-1}.
\]
Notice that $x_{\pm}(a=\frac{2\lambda}{1-\lambda})=1$ and that $x_{-}(a\rightarrow\infty)\rightarrow\lambda$,
$x_{+}(a\rightarrow\infty)\rightarrow1/\lambda$. The function 
\[
a\mapsto y(a)=\frac{a+(a+2)\lambda^{2}}{2\lambda(a+1)},\ a>-1
\]
is increasing for $\lambda\in(0,1)$. The functions
$y\mapsto y\pm\sqrt{y^{2}-1}$ are, respectively, increasing/ decreasing
for $y>1$. It follows that $x_{-}(a)$ is decreasing and $x_{+}(a)$
is increasing for $a>\frac{2\lambda}{1-\lambda}$. Consequently $x_{-}\in(\lambda,1)$
and $x_{+}\in(1,1/\lambda)$. Choose $x^{*}=x_{-}$ and notice that
the second derivative of $x\mapsto x^{a+1}\frac{1-\lambda x}{x-\lambda}$ is strictly positive at $x^{*}$;
Moreover $\lim_{x\rightarrow\lambda^{+}}x^{a+1}\frac{1-\lambda x}{x-\lambda}=+\infty$
and $x^{a+1}\frac{1-\lambda x}{x-\lambda}\Big|_{x=1}=1$. Taken together
these properties are only possible if $\left(x^{*}\right)^{a+1}\frac{1-\lambda x^{*}}{x^{*}-\lambda}<1$.
\item Suppose that $a<-\frac{2\lambda}{1+\lambda}$. The derivative of $x\mapsto x^{a+1}\frac{1+\lambda x}{x+\lambda}$
has zeros at 
\[
x_{\pm}=-\frac{a+(a+2)\lambda^{2}}{2\lambda(a+1)}\pm\sqrt{\left(\frac{a+(a+2)\lambda^{2}}{2\lambda(a+1)}\right)^{2}-1}.
\]
Notice that $x_{\pm}(a=-\frac{2\lambda}{1+\lambda})=1$ and that $x_{-}(a\rightarrow-1)\rightarrow0$,
$x_{+}(a\rightarrow-1)\rightarrow\infty$. The function \[
a\mapsto y(a)=\frac{a+(a+2)\lambda^{2}}{2\lambda(a+1)}, \ a>-1
\]
is increasing for $\lambda\in(0,1)$. The functions $y\mapsto-y\pm\sqrt{y^{2}-1}$ are,
respectively, decreasing/ increasing for $y<-1$. It follows that $x_{-}(a)$ is increasing and $x_{+}(a)$
is decreasing on $(-1,-\frac{2\lambda}{1+\lambda})$.
Consequently $x_{-}\in(0,1)$ and $x_{+}\in(1,\infty)$. Fix $\epsilon>0$, choose $x^{*}=\min\{x_{+},1/\lambda-\epsilon\}$ and notice that the second derivative of $x\mapsto x^{a+1}\frac{1+\lambda x}{x+\lambda}$ is strictly positive at $x_{+}$; Moreover $\lim_{x\rightarrow\infty}x^{a+1}\frac{1+\lambda x}{x+\lambda}=+\infty$
and $x^{a+1}\frac{1+\lambda x}{x+\lambda}\Big|_{x=1}=1$. Taken together
these properties are only possible if $\left(x^{*}\right)^{a+1}\frac{1+\lambda x^{*}}{x^{*}+\lambda}<1$.

\end{enumerate}
%\begin{figure}[ht!]
%\centering
%\includegraphics[scale=0.5]{plotY.png}
%\caption{Plot of functions $f_+(y)= y+\sqrt{y^2-1}$ (red) and
%$f_-(y)= y-\sqrt{y^2-1}$ (blue). \label{plot}}
%\end{figure}
\end{proof}
\begin{proof}[Proof of Proposition~\ref{Second_thm}, point 2] Since $an+\alpha$ is not an integer both the Fourier and the Laplace integral contribute. The generalized Fourier integral is estimated above. 
\begin{enumerate}[(a)]
\item Suppose that $a>\frac{2\lambda}{1-\lambda}$. By Lemma~\ref{trivial}, point (1), the function $g:\:t\mapsto t^{(a+1)}\frac{1+\lambda t}{t+\lambda}$
is strictly increasing on $(0,1]$ and $g(t)\leq1$ for $t\in(0,1]$.
Consequently, with $x^{*}$ chosen as in~{the proof of Proposition~\ref{Second_thm}, point
(1a)}, it holds that $\left(x^{*}\right)^{a+1}\frac{1+\lambda x^{*}}{x^{*}+\lambda}<1$.
It follows that
\begin{align*}
\int_{0}^{x^{*}}\frac{(1+\lambda t)^{\beta}t^{\alpha}}{t+\lambda}\left(t^{a+1}\frac{1+\lambda t}{t+\lambda}\right)^{n}{\rm d}t\leq\frac{(x^{*})^{\alpha+1}(1+\lambda x^{*})^{\beta+1}}{(x^{*}-\lambda)(1-\lambda x^{*})}\left(\left(x^{*}\right)^{a+1}\frac{1+\lambda x^{*}}{x^{*}+\lambda}\right)^{n}
\end{align*}
which, combined with the estimate for the generalized Fourier integral,
proves (2). 
\item Suppose that $-1<a<-2\lambda/(1+\lambda)$. Choose $x^{*}$ as in~{the proof of Proposition~\ref{Second_thm}, point (1b)}. By Proposition~\ref{Laplace}, point 3, and the same reasoning as in \emph{(a)} the Laplace integral increases exponentially in magnitude. At the same time the Fourier integral decays exponentially. This proves the assertion in case the sequence $(an+\alpha)_n$ is separate from integers. In general we refer to formula~\eqref{eq:new_asymp_formula}, see the proof of Theorem~\ref{gbar} point 4) below.
\end{enumerate}
\end{proof}

\section{Proof of Theorem \ref{gbar} and Comments}
\label{proofSection}
By Lemma~\ref{striking} the asymptotic behavior of the JPVFPs can be obtained as a sum of the asymptotic formulas for the Fourier and Laplace-type integrals. Depending on the region of parameters either the one or the other (or both) provide the leading order contribution, which results in the diverse asymptotic behavior in regions 1)-4). It should be noticed that it is still a delicate question, which summand in the expansion of the JPVFPs actually provides the dominating contribution. Consider for instance the first order term in point 1),
\begin{align}\label{four}
\sqrt{\frac{2}{n\pi}}\cos\left(((n+1)h(\varphi_{+})+(\alpha-a)\varphi_{+}+(\beta-1)\psi{+}\frac{\pi}{4}\right),
\end{align}
which is clearly $\cO(1/n^{1/2})$, while the next term is $\cO(1/n^{3/2})$. It is, in principle, possible that a subsequence exists such that the above is even $\cO(1/n^{1/2})$. We illustrate this phenomenon with a brief lemma.
\begin{lem}\label{dio}
Let $\gamma>0$ be fixed. There exists a sequence $s_n$ of natural numbers so that $\abs{\cos(\gamma s_n + \frac{\pi}{2})}\leq\frac{K}{s_n}$.
\end{lem}
\begin{proof}[Proof of Lemma~\ref{dio}]
We show that the sequence $a_n=n\cos(\gamma n + \frac{\pi}{2})$ has a convergent subsequence. Fix arbitrary $n$ and let, by Dirichlet's theorem on Diophantine approximation, 
$p_n, q_n\in\mathbb{N}$ be chosen so that
\begin{align*}
\Abs{q_n\frac{\pi}{\gamma} - p_n}<\frac{1}{n},\quad q_n\leq n.
\end{align*}
The subsequence $a_{p_n}$ is bounded because
\begin{align*}
\abs{a_{p_n}}=\Abs{p_n\cos(\gamma p_n  + \frac{\pi}{2})}
= p_n\Abs{\sin\left(\gamma p_n - q_n \pi\right)}
\leq \left(n\frac{\pi}{\gamma}+\frac{1}{n}\right)\frac{\gamma}{n}\leq \pi + \gamma.
\end{align*}
By the Bolzano-Weierstrass theorem $a_{p_n}$ has a convergent subsequence $a_{s_n}$.
\end{proof}
%
%
%For completeness we write down the short proof of Theorem~\ref{gbar}.
%
%
%
\begin{proof}[Proof of Theorem~\ref{gbar}]
\leavevmode
\begin{enumerate}
\item Suppose $a\in(-\frac{2\lambda}{1+\lambda},\,\frac{2\lambda}{1-\lambda})$ and
choose $x=1$ in Lemma~\ref{striking}. Comparing the second order term of the generalized Fourier integral in Proposition~\ref{thm:delta_negative}, point (1) to the first order term of the Laplace integral in Proposition~\ref{Laplace}, point (1),
we observe an exact cancellation. We are therefore left with the first order term of the generalized Fourier integral, which is proportional to~\eqref{four} and an additional contribution of $\cO(1/n^{3/2})$.

\item Suppose $a=\frac{2\lambda}{1-\lambda}$ and choose $x=1$ in Lemma~\ref{striking}. By Proposition~\ref{Laplace}, point (1), the contribution of the Laplace-type integral for $a>-\frac{2\lambda}{1+\lambda}$ is $\cO(1/n)$. By Proposition~\ref{thm:delta_negative}, point (2), the leading contribution to the generalized Fourier integral is $\Theta(1/n^{1/3})$. As a consequence only the Fourier integral determines the leading order behavior in this region.
Suppose $a=-\frac{2\lambda}{1+\lambda}$ and choose $x=1$ in Lemma~\ref{striking}. By Proposition~\ref{Laplace}, point (2), and by Proposition~\ref{thm:delta_negative}, point (2), both integrals are of order $\cO(1/n^{1/3})$. Summing the contributions according to Lemma~\ref{striking} yields the formula claimed in the theorem.

\item If $an+\alpha$ is an integer then the Laplace-type integral does not contribute. The exponential decay of the Fourier-type integral in the region $a\in(-1,\frac{-2\lambda}{1+\lambda})\cup(\frac{2\lambda}{1-\lambda},\infty)$ has been demonstrated as part of Proposition~\ref{Second_thm}, which proves the statement of the theorem. For completeness let us mention that an application of the well-established saddle point method
~\cite[Chapter II]{RW} to the Fourier-type integral yields
\begin{align*}
 & \lambda^{an+\alpha}(1-\lambda^{2})^{\beta}P_{n}^{(an+\alpha,\,\beta)}(1-2\lambda^{2})\\
 & =\sqrt{\frac{1}{2n\pi(a+1)}}\frac{1}{\Delta^{1/4}}\frac{z_{-}^{\alpha+1}(1-\lambda z_{-})^{\beta}}{z_{-}-\lambda}\left(\frac{z_{-}^{a+1}(1-\lambda z_{-})}{z_{-}-\lambda}\right)^{n}\left(1+\cO\left(\frac{1}{n^{1/2}}\right)\right),
\end{align*}
where 
\[
\Delta=\left(a+\frac{2\lambda}{1+\lambda}\right)\left(a-\frac{2\lambda}{1-\lambda}\right),
\]
\[
z_{-}=\frac{a+a\lambda^{2}+2\lambda^{2}}{2\lambda(a+1)}-\sqrt{\left(\frac{a+a\lambda^{2}+2\lambda^{2}}{2\lambda(a+1)}\right)^{2}-1},
\]
and $\Abs{\frac{z_{-}^{a+1}(1-\lambda z_{-})}{z_{-}-\lambda}}<1$.
This is a verbatim adaptation of \cite[Theorem 1, (2) and (4)]{BKZ}.

\item If $an+\alpha$ is not an integer then the Laplace-type integral does contribute.
If $a>\frac{2\lambda}{1-\lambda}$ it can be checked that the Laplace integral decays faster than the Fourier integral and the above formula remains valid. For $a\in(-1,\frac{-2\lambda}{1+\lambda})$ the Laplace integral grows exponentially, while the Fourier integral decays exponentially, which proves the statement of the theorem. For completeness let us mention that Proposition~\ref{Laplace}, point 3) entails that 
\begin{align}
 & \lambda^{an+\alpha}(1-\lambda^{2})^{\beta}P_{n}^{(an+\alpha,\,\beta)}(1-2\lambda^{2})\label{eq:new_asymp_formula}\\
 & =-\frac{{\sin\left(\pi(an+\alpha)\right)}}{\pi}(1+\lambda t_{-})^{\beta+1}t_{-}^{\alpha}\left(g(t_{-})\right)^{n}\sqrt{\frac{2\pi t_{-}}{n\lambda\left((1+\lambda t_{-})^{2}-(\lambda+t_{-})^{2}\right)}}\nonumber \\
 & +\cO\left(\frac{(g(t_{-}))^{n}}{n}\right),\nonumber 
\end{align}
where $t_{-}$ is given in \eqref{eq:tplusminus}, $g(t_{-})>1$.
The precise asymptotic behavior of this quantity depends on the specific subsequence of natural numbers, as the proximity of this subsequence to the zeros of the sine term potentially tempers the rate of growth.
\end{enumerate}
\end{proof}
%
%
%
%
%\newpage

\section{Numerical experiments}

\label{numericalExperiments}\label{numericalExperiments}
We illustrate the findings of our main Theorem~\ref{gbar} with numerical plots in the respective parameter regions. In all cases the numerical experiments are consistent with the theoretical work. Results were obtained in MAPLE2020 and using the NumPy and Matplotlib packages of Python.

Figure~\ref{TheoremRegion1} illustrates Theorem~\ref{gbar}, point (1). JPFVPs of varying degree $n$ are evaluated numerically and compared to our approximation for two sets of parameters.

Figure~\ref{TheoremRegion2a} illustrates Theorem~\ref{gbar}, point (2) at the right boundary, $a=2\lambda/(1-\lambda)$.

Figure~\ref{TheoremRegion2b} illustrates Theorem~\ref{gbar}, point (2) at the left boundary, $a=-2\lambda/(1+\lambda)$.

Figure~\ref{TheoremRegion3} illustrates the exponential region $a\in(-1,\frac{-2\lambda}{1+\lambda})\cup(\frac{2\lambda}{1-\lambda},\infty)$ of Theorem~\ref{gbar}, point (3). For two sets of parameters the JPFVPs are evaluated and displayed on a subsequence of degrees, where $an+\alpha$ is an integer.

Figure~\ref{TheoremRegion4} illustrates the exponential region $a\in(-1,\frac{-2\lambda}{1+\lambda})\cup(\frac{2\lambda}{1-\lambda},\infty)$ of Theorem~\ref{gbar}, point (4). For three sets of parameters the JPFVPs are evaluated and displayed on a subsequence of degrees, where $an+\alpha$ is not an integer. In accordance with the claim of the theorem exponential decay is observed if $a>\frac{2\lambda}{1-\lambda}$. However if $a<\frac{-2\lambda}{1+\lambda}$ the JPFVPs show exponential growth.
%%%%%%%%%%%%%%%%%%%%%%%%%%%%%%%%%%%%%%%%%%%%%%%%%%%%%%%%
%%%%%%%%%%%%%%%%%%%%%%%%%%%%%%%%%%%%%%%%%%%%%%%%%%%%%%%%
%%%%%%%%%%%%%%%%%%%%%%%%%%%%%%%%%%%%%%%%%%%%%%%%%%%%%%%%
%point 1 
\begin{figure}[h]
\centering
\begin{subfigure}{\textwidth}
\centering
\includegraphics[width=0.9\linewidth]{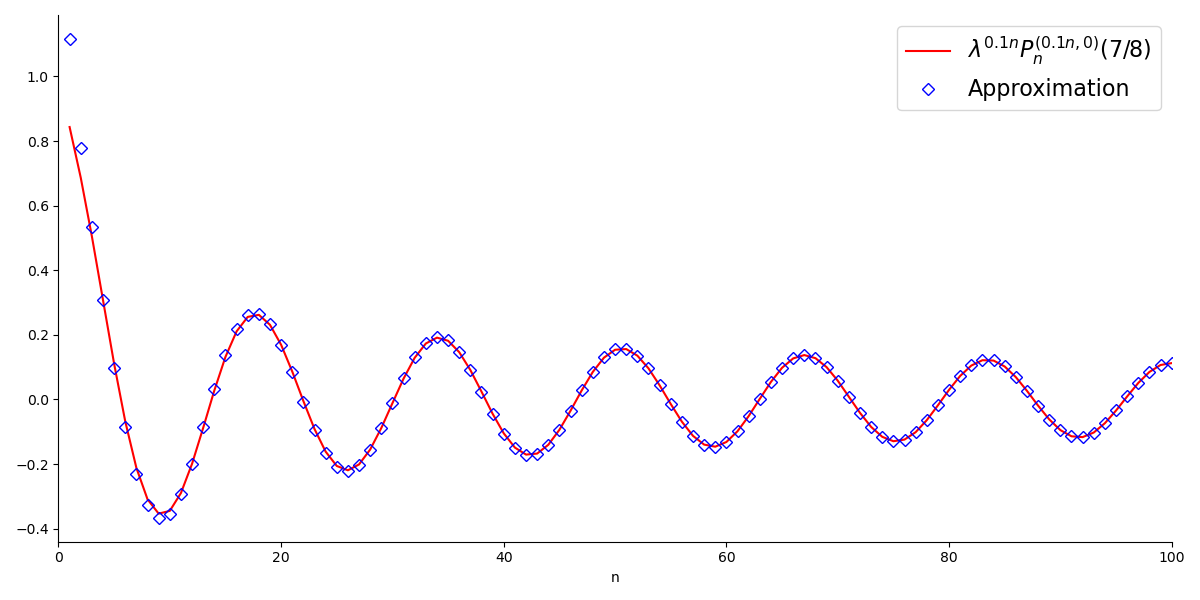}
\caption{$\lambda=\frac{1}{4}$, $a=0.1$, $\alpha=\beta=0$}
\end{subfigure}
\begin{subfigure}{\textwidth}
\centering
\includegraphics[width=0.9\linewidth]{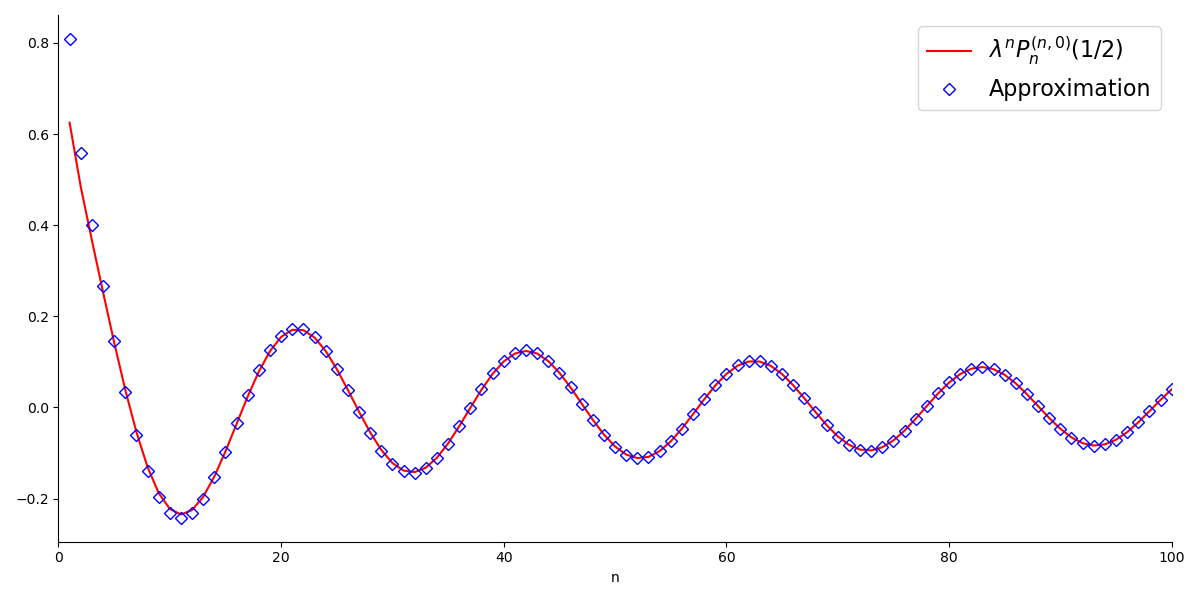}
\caption{$\lambda=\frac{1}{2}$, $a=1$, $\alpha=\beta=0$}
\end{subfigure}
\centering
\caption{Illustration of Theorem \ref{gbar}, point (1). Numerical computation of $n$-dependence of JPVFPs and approximation for two sets of parameters.}
\label{TheoremRegion1}
\end{figure}
\newpage

%%%%%%%%%%%%%%%%%%%%%%%%%%%%%%%%%%%%%%%%%%%%%%%%%%%%%%%%
%%%%%%%%%%%%%%%%%%%%%%%%%%%%%%%%%%%%%%%%%%%%%%%%%%%%%%%%
%point 2 a)
\begin{figure}[h]
\centering
\begin{subfigure}{\textwidth}
\centering
\includegraphics[width=0.9\linewidth]{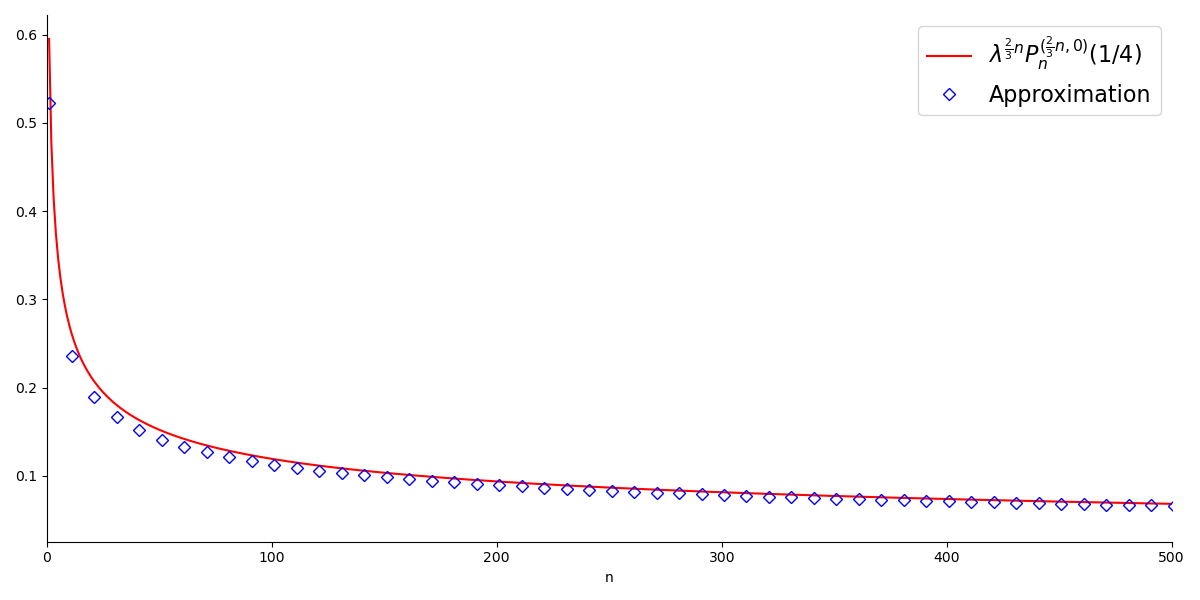}
\caption{$\lambda=\frac{1}{4}$, $a=\frac{2}{3}$, $\alpha=\beta=0$}
\end{subfigure}
\begin{subfigure}{\textwidth}
\centering
\includegraphics[width=0.9\linewidth]{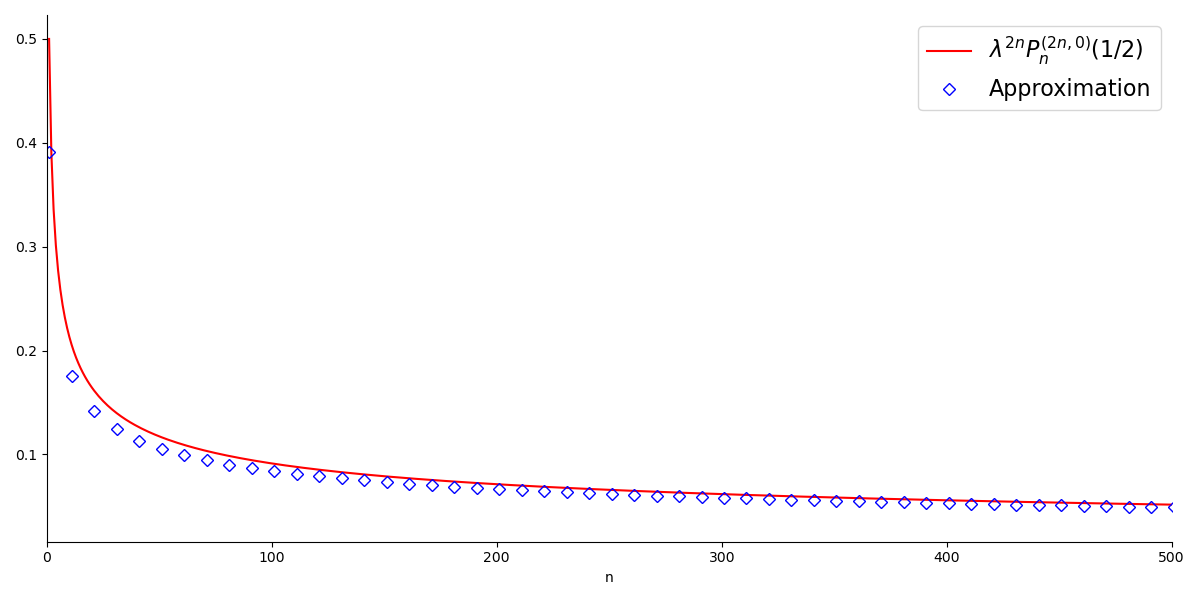}
\caption{$\lambda=\frac{1}{2}$, $a=2$, $\alpha=\beta=0$}
\end{subfigure}
\centering
\caption{Illustration of Theorem \ref{gbar}, point (2). Numerical computation of $n$-dependence of JPVFPs and approximation at the right boundary. In case of the approximation only every 10th data point is displayed.}
\label{TheoremRegion2a}
\end{figure}
\newpage

%%%%%%%%%%%%%%%%%%%%%%%%%%%%%%%%%%%%%%%%%%%%%%%%%%%%%%%%
%%%%%%%%%%%%%%%%%%%%%%%%%%%%%%%%%%%%%%%%%%%%%%%%%%%%%%%%
%point 2 b)
\begin{figure}[h]
\centering
\begin{subfigure}{\textwidth}
\centering
\includegraphics[width=0.9\linewidth]{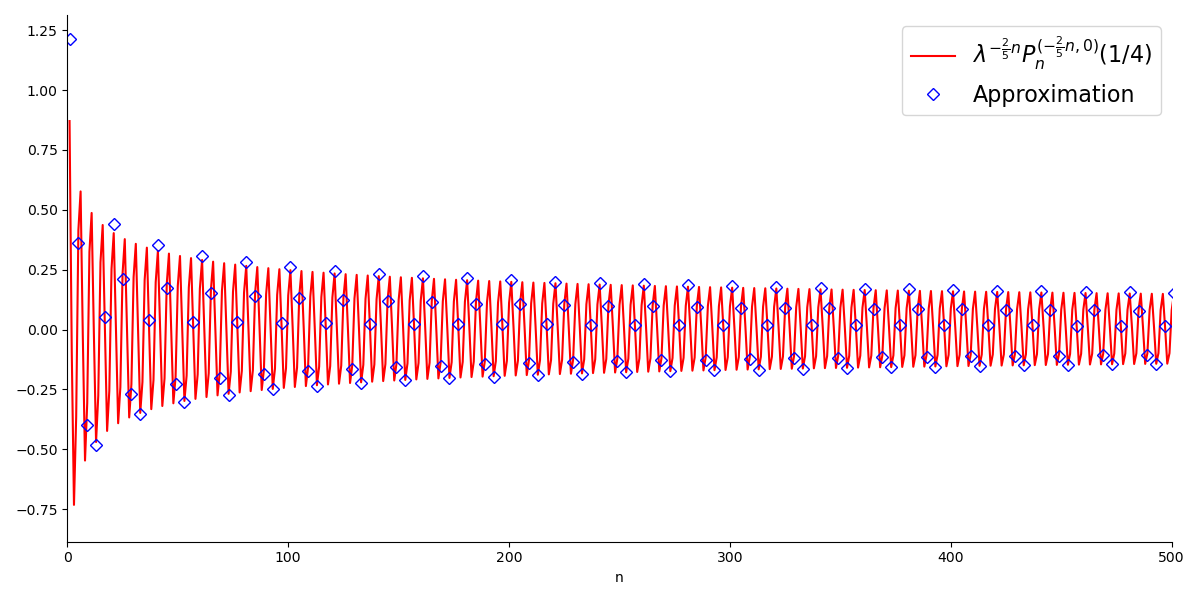}
\caption{$\lambda=\frac{1}{4}$, $a=-\frac{2}{5}$, $\alpha=\beta=0$}
\end{subfigure}
\begin{subfigure}{\textwidth}
\centering
\includegraphics[width=0.9\linewidth]{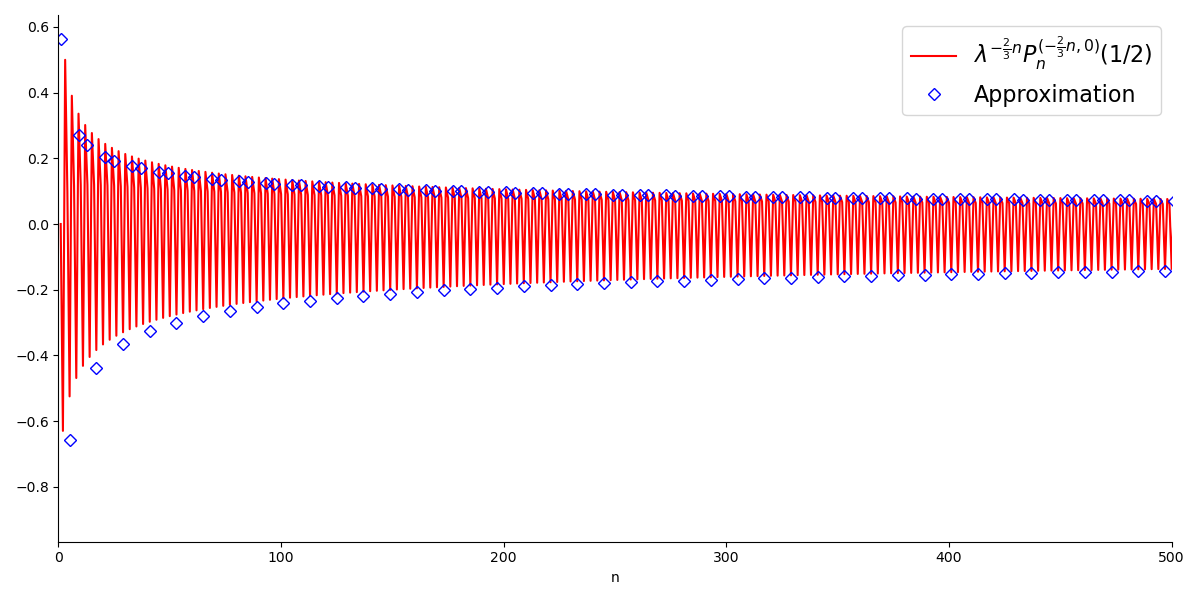}
\caption{$\lambda=\frac{1}{2}$, $a=-\frac{2}{3}$, $\alpha=\beta=0$}
\end{subfigure}
\centering
\caption{Illustration of Theorem \ref{gbar}, point (2). Numerical computation of $n$-dependence of JPVFPs and approximation at the left boundary. In case of the approximation only every 4th data point is displayed.}
\label{TheoremRegion2b}
\end{figure}
\newpage

%%%%%%%%%%%%%%%%%%%%%%%%%%%%%%%%%%%%%%%%%%%%%%%%%%%%%%%%
%%%%%%%%%%%%%%%%%%%%%%%%%%%%%%%%%%%%%%%%%%%%%%%%%%%%%%%%
%point 3, region of exponential decay.

\begin{figure}[h]
\centering
\begin{subfigure}{\textwidth}
\centering
\includegraphics[width=0.9\linewidth]{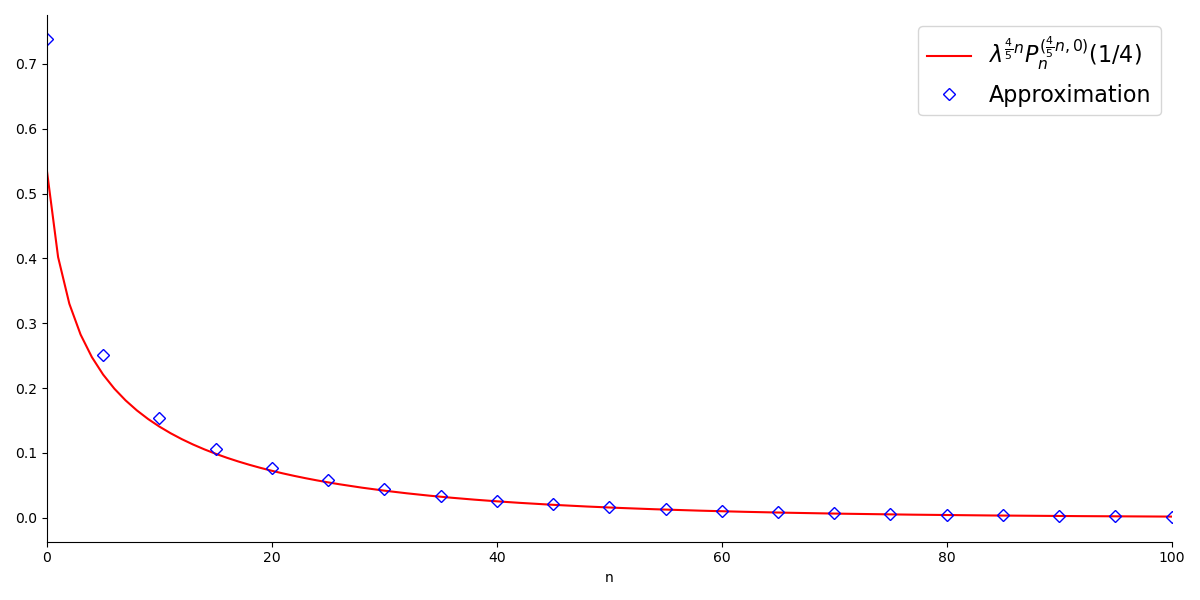}
\caption{$\lambda=\frac{1}{4}$, $a=\frac{4}{5}>\frac{2}{3}$, $\alpha=\beta=0$}
\end{subfigure}
\begin{subfigure}{\textwidth}
\centering
\includegraphics[width=0.9\linewidth]{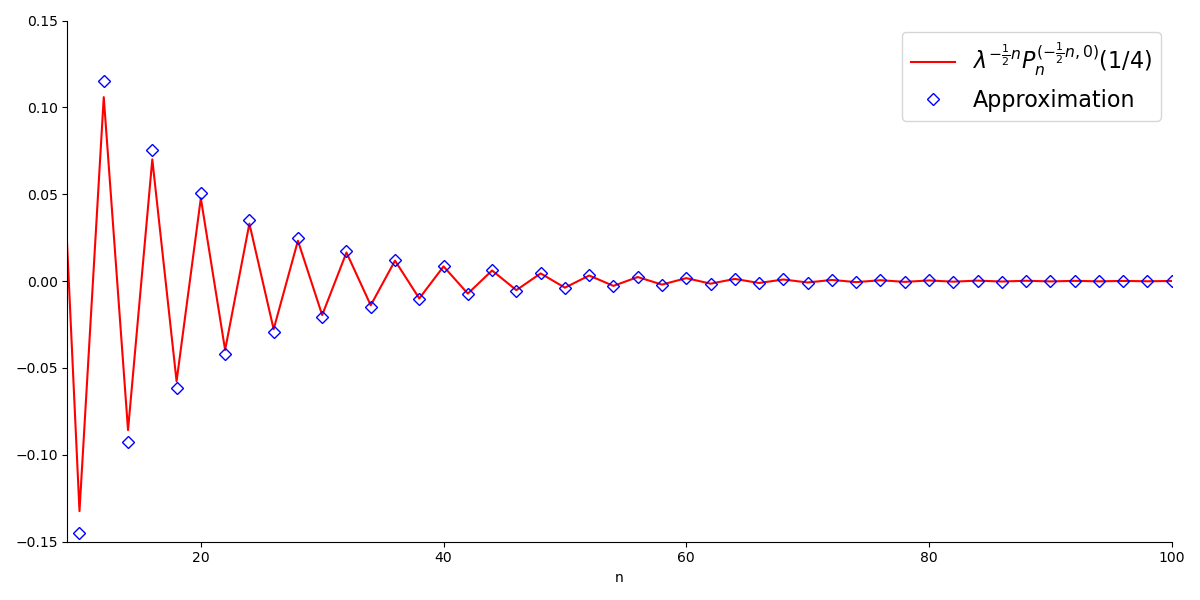}
\caption{$\lambda=\frac{1}{4}$, $a=-\frac{1}{2}<-\frac{2}{5}$, $\alpha=\beta=0$}
\end{subfigure}
\centering
\caption{Illustration of Theorem \ref{gbar}, point (3). Numerical computation of $n$-dependence of JPVFPs and approximation in the exponential region. Data points are displayed only for such $n$ that $an+\alpha$ is an integer.}
\label{TheoremRegion3}
\end{figure}
\newpage

%%%%%%%%%%%%%%%%%%%%%%%%%%%%%%%%%%%%%%%%%%%%%%%%%%%%%%%%
%%%%%%%%%%%%%%%%%%%%%%%%%%%%%%%%%%%%%%%%%%%%%%%%%%%%%%%%
%point 4
\begin{figure}[h]
\centering
\begin{subfigure}{\textwidth}
\centering
\includegraphics[width=0.9\linewidth]{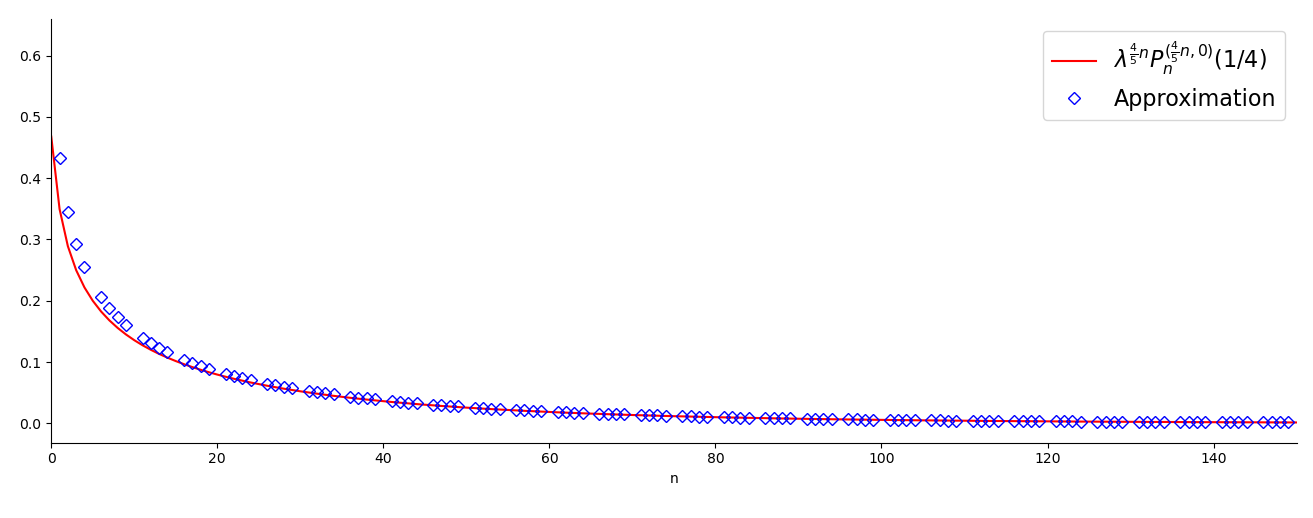}
\caption{$\lambda=\frac{1}{4}$, $a=\frac{4}{5}>\frac{2}{3}$, $\alpha=\beta=0$}
\end{subfigure}
\centering
\begin{subfigure}{\textwidth}
\centering
\includegraphics[width=0.9\linewidth]{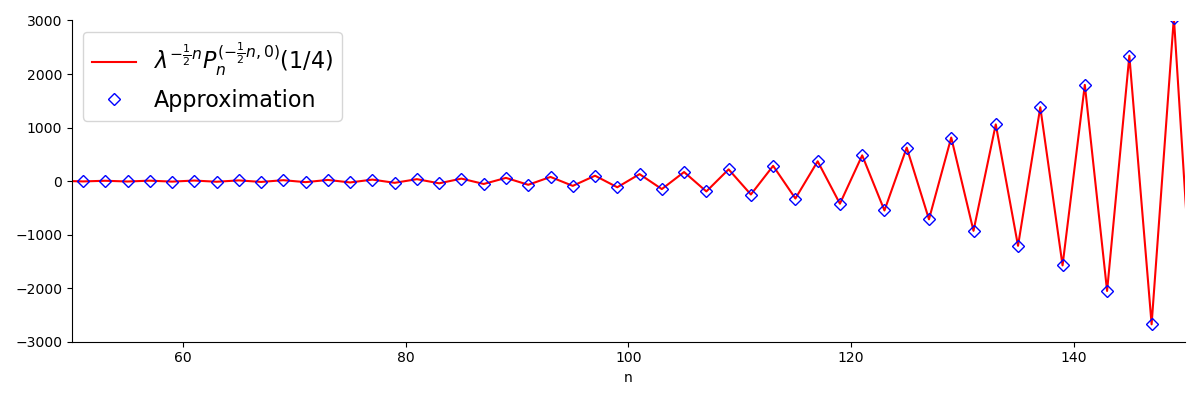}
\caption{$\lambda=\frac{1}{4}$, $a=-\frac{1}{2}<-\frac{2}{5}$, $\alpha=\beta=0$}
\end{subfigure}
\begin{subfigure}{\textwidth}
\centering
\includegraphics[width=0.9\linewidth]{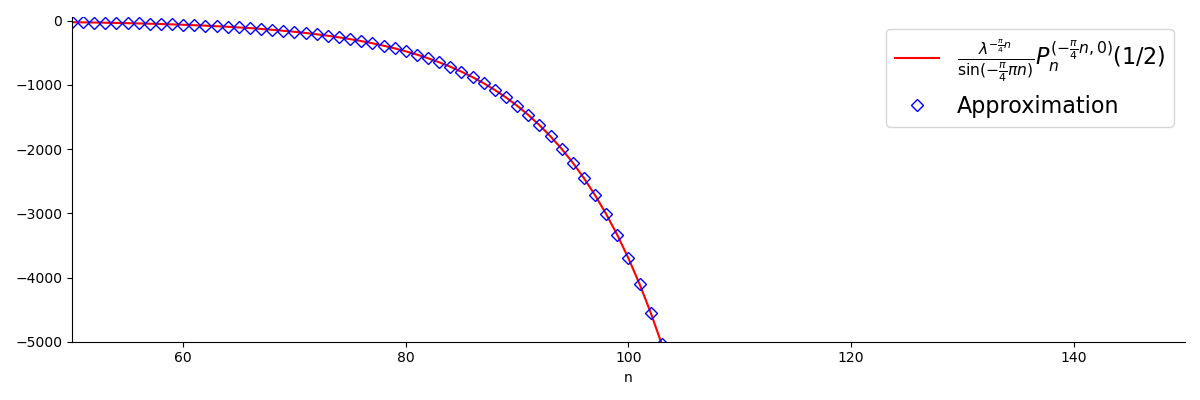}
\caption{$\lambda=\frac{1}{2}$, $a=-\frac{\pi}{4}<-\frac{2}{3}$, $\alpha=\beta=0$}
\end{subfigure}
\centering
\caption{Illustration of Theorem \ref{gbar}, point (4). Numerical computation of $n$-dependence of JPVFPs and approximation in the exponential region. Data points are displayed only for such $n$ that $an+\alpha$ is not an integer.}
\label{TheoremRegion4}
\end{figure}
%\newpage

\newpage
\section{Extensions}

This section illustrates two possible extensions of our work. \label{sec:Final-remarks} 

\subsection{Jacobi polynomials with varying parameters}

The proof of Lemma~\ref{striking} can be adapted to the more general
JPVPs by integrating the function 
\[
f:z\mapsto z^{(a+1)n+\alpha}\frac{(1-\lambda z)^{(b+1)n+\beta}}{(z-\lambda)^{n+1}}
\]
along the same contour as in Lemma~\ref{striking}. One easily obtains
an integral representation for the JPVPs of the form 
\begin{align*}
 & \lambda^{an+\alpha}(1-\lambda^{2})^{bn+\beta}P_{n}^{(an+\alpha,\,bn+\beta)}(1-2\lambda^{2})\\
 & =\frac{1}{\pi}\Re\left\{ \int_{0}^{\pi}{z^{\alpha+1}\frac{(1-\lambda z)^{\beta}}{z-\lambda}\bigg(z^{a+1}\frac{(1-\lambda z)^{b+1}}{z-\lambda}\bigg)^{n}}\Bigg|_{z={x}e^{i\varphi}}\d\varphi\right\} \\
 & -\frac{{\sin\left(\pi(\alpha+an)\right)}}{\pi}\int_{0}^{{x}}\frac{(1+\lambda t)^{\beta}t^{\alpha}}{t+\lambda}\left(t^{(a+1)}\frac{(1+\lambda t)^{b+1}}{t+\lambda}\right)^{n}{\rm d}t,
\end{align*}
where $x\in(\lambda,\,1/\lambda)$. It should be mentioned that in
this case the asymptotic analysis of the first integral is significantly
more involved. Some asymptotic formulas can be obtained by applying the {method
of steepest descent} along the lines of~\cite[Section 2]{RW}, \cite[Chapter 7]{BH}.
This leads to the cases 
\begin{enumerate}
\item $a\in\left(\frac{(1-\lambda\sqrt{b+1})^{2}}{1-\lambda^{2}}-1,\,\frac{(1+\lambda\sqrt{b+1})^{2}}{1-\lambda^{2}}-1\right)$, 
\item $a\in\left\{ \frac{(1-\lambda\sqrt{b+1})^{2}}{1-\lambda^{2}}-1,\,\frac{(1+\lambda\sqrt{b+1})^{2}}{1-\lambda^{2}}-1\right\} ,$ 
\item $a\in\left(\frac{(1+\lambda\sqrt{b+1})^{2}}{1-\lambda^{2}}-1,\,\infty\right),$ 
\item $a\in\left(-1,\,\frac{(1-\lambda\sqrt{b+1})^{2}}{1-\lambda^{2}}-1\right)$
and $an+\alpha$ is an integer, 
\item $a\in\left(-1,\,\frac{(1-\lambda\sqrt{b+1})^{2}}{1-\lambda^{2}}-1\right)$
and $an+\alpha$ is not an integer. 
\end{enumerate}
For instance if $a\in\left(\frac{(1-\lambda\sqrt{b+1})^{2}}{1-\lambda^{2}}-1,\,\frac{(1+\lambda\sqrt{b+1})^{2}}{1-\lambda^{2}}-1\right)$
then
\begin{align*}
&\lambda^{an+\alpha}(1-\lambda^{2})^{bn+\beta}P_{n}^{(an+\alpha,\,bn+\beta)}(1-2\lambda^{2})=\\
&\sqrt{\frac{2}{n\pi}}\Re\left\{ \left(G(z_{+})\left(z_{+}^{2}iH''(z_{+})\right)^{-1/2}\right)e^{nH(z_{+})+i\frac{\pi}{4}}\right\} \left(1+\cO\left(\frac{1}{n^{1/2}}\right)\right),
\end{align*}
where 
\[
G(z)=\frac{z^{\alpha+1}(1-\lambda z)^{\beta}}{z-\lambda},\qquad H(z)=\Log\left(\frac{z^{a+1}(1-\lambda z)^{b+1}}{z-\lambda}\right)
\]
and 
\[
z_{\pm}=\frac{\lambda^{2}(a+b+2)+a}{2\lambda(a+b+1)}\pm\sqrt{\left(\frac{\lambda^{2}(a+b+2)+a}{2\lambda(a+b+1)}\right)^{2}-\frac{a+1}{a+b+1}}
\]
satisfy $H'(z_{\pm})=0$ and $H''(z_{\pm})\neq0$, which generalizes
Theorem \ref{gbar} point (1). This formula coincides with~\cite[Proposition 6.1]{FFN}.
Similar to Theorem~\ref{gbar}, point 3), if $a\in\left(\frac{(1+\lambda\sqrt{b+1})^{2}}{1-\lambda^{2}}-1,\,\infty\right)$
and $a+b+1>0$ then the Laplace-type integral does not contribute,
leading to 
\begin{align*}
 & \lambda^{an+\alpha}(1-\lambda^{2})^{bn+\beta}P_{n}^{(an+\alpha,\,bn+\beta)}(1-2\lambda^{2})\\
 & =\sqrt{\frac{1}{2n\pi H''(z_{-})}}\frac{G(z_{-})}{z_{-}}\left(\frac{z_{-}^{a+1}(1-\lambda z_{-})^{b+1}}{z_{-}-\lambda}\right)^{n}\left(1+\cO\left(\frac{1}{n^{1/2}}\right)\right).
\end{align*}
However, if $a\in\left(-1,\,\frac{(1-\lambda\sqrt{b+1})^{2}}{1-\lambda^{2}}-1\right)$
the condition that {$an+\alpha$ is an integer} plays a similar role as for $b=0$. If $a\in\left(-1,\,\frac{(1-\lambda\sqrt{b+1})^{2}}{1-\lambda^{2}}-1\right)$,
$a+b+1>0$ and $an+\alpha$ is an integer then 
\begin{align*}
 & \lambda^{an+\alpha}(1-\lambda^{2})^{bn+\beta}P_{n}^{(an+\alpha,\,bn+\beta)}(1-2\lambda^{2})\\
 & =-\sqrt{\frac{1}{2n\pi H''(z_{-})}}\frac{z_{-}^{(a+1)n+\alpha}(1-\lambda z_{-})^{(b+1)n+\beta}}{(z_{-}-\lambda)^{n+1}}\left(1+\cO\left(\frac{1}{n^{1/2}}\right)\right).
\end{align*}

\subsection{Uniform asymptotic expansions}

\label{uniform} The theory of uniform asymptotic expansions is concerned
with asymptotic expansions that hold when a parameter varies throughout different ranges of asymptotic behavior. We describe
briefly the methodology to obtain uniform expansions for JPVFPs when
$a$ passes over the critical boundaries $\{-\frac{2\lambda}{1+\lambda},\frac{2\lambda}{1-\lambda}\}$.
Luckily even in this situation standard tools are available \cite[Section 2.3]{Bo}, \cite[Section 9.2]{BH},
\cite[p. 366--372]{RW}, which are all based on the so-called \emph{uniform
method of steepest decent}~\cite{CFU}. 
For brevity consider the Fourier integral, see Section~\ref{sub:The-stationary-phase}.
When $a$ approaches the boundary, the radius of convergence of the
asymptotic expansion goes to $0$. Indeed when $a$ varies in $\mathbb{R}\setminus\{-\frac{2\lambda}{1+\lambda},\frac{2\lambda}{1-\lambda}\}$
the saddle points $z_{\pm}$ of $h_{a}$ are of order one, but when
$a$ approaches the critical boundary $z_{\pm}$ coalesce to saddle
points of order $2$. If $a$ approaches the boundary from the inside
the two saddle points $z_{\pm}$ remain on the unit circle $\partial\mathbb{D}$.
But if $a$ approaches the boundary from the outside the saddle
points $z_{\pm}$ move along the real line. While in the former situation
$z_{\pm}$ automatically lie on the contour of integration, in the latter case the contour
of integration is deformed such that the new contour passes through
the saddle points $z_{\pm}$ and $1$. To simplify the dependence
of $z_{\pm}$ on $a$ a change of variable is applied via a locally
one-to-one transformation. This is made precise in the below proposition, where (overriding previous notation)
\[
f(z)=f_{a}(z)={\rm Log}\left(\frac{z^{a+1}(1-\lambda z)}{z-\lambda}\right).
\]

\begin{prop}[{{\cite[Proposition 16]{SZ1}}}]
\label{cubic} For $a$ in a small neighborhood
of $\frac{2\lambda}{1-\lambda}$ the cubic transformation 
\[
f_{a}(z)=-\frac{t^{3}}{3}+\gamma^{2}t
\]
with 
\[
\gamma^{2}=\frac{(a-\frac{2\lambda}{1-\lambda})(1-\lambda)}{\left(\lambda(1+\lambda)\right)^{1/3}}+{\scriptstyle\mathcal{O}}\left(a-\frac{2\lambda}{1-\lambda}\right),
\]
has exactly one branch $t=t(z,a)$ that can be expanded into a power
series in $z$ with coefficients that are continuous in $a$. On this
branch the points $z=z_{\pm}$ correspond to $t=\pm\gamma$. The mapping
of $z$ to $t$ is one-to-one on a small neighborhood of $z=1$ containing
$z_{+}$ and $z_{-}$. 
\end{prop}

This is an immediate corollary of~\cite{CFU}. A proof can be found
in the appendix of \cite{SZ2}. To determine the asymptotics of the
Fourier integral over a new suitable contour $\cC$
we apply the transformation to a neighborhood of $z_{0}=1$. This
yields a uniform expansion of the integral in terms of the Airy function
$Ai$. For real arguments the latter can be defined as 
\[
Ai(x)=\frac{1}{\pi}\int_{0}^{\infty}\cos\left(\frac{t^{3}}{3}+xt\right)\d t.
\]
For large negative arguments the Airy function shows oscillatory behavior
\[
Ai(-x)=\frac{1}{x^{1/4}\sqrt{\pi}}\left(\cos\left(\frac{2}{3}x^{3/2}-\frac{\pi}{4}\right)+{\scriptstyle \mathcal{O}}(1)\right),\qquad x\rightarrow+\infty,
\]
\[
Ai'(-x)=\frac{x^{1/4}}{\sqrt{\pi}}\left(\sin\left(\frac{2}{3}x^{3/2}-\frac{\pi}{4}\right)+{\scriptstyle \mathcal{O}}(1)\right),\qquad x\rightarrow+\infty,
\]
%
%
%\[
%Ai(-x)\sim\frac{1}{x^{1/4}\sqrt{\pi}}\cos\left(\frac{2}{3}x^{3/2}-\frac{\pi}{4}\right),\ Ai'(-x)\sim\frac{x^{1/4}}{\sqrt{\pi}}\sin\left(\frac{2}{3}x^{3/2}-\frac{\pi}{4}\right),\ x\rightarrow+\infty,
%\]
and exponential behavior for large positive arguments 
\[
Ai(x)\sim\frac{1}{2x^{1/4}\sqrt{\pi}}\exp\left(-\frac{2}{3}x^{3/2}\right),\ Ai'(x)\sim-\frac{x^{1/4}}{2\sqrt{\pi}}\exp\left(-\frac{2}{3}x^{3/2}\right),\ x\rightarrow+\infty.
\]
Following the procedure described in \cite[p.~371--375]{BH} yields an asymptotic expansion of 
$\int_{\partial\mathbb{D}}g(z)e^{nf(z)}\frac{\d z}{z}$ for $a$ in a neighborhood of $\frac{2\lambda}{1-\lambda}$ in terms of $Ai$ and $Ai'$. Finally, if $an+\alpha$ is not an integer the Laplace integral must
be taken into account. In this case \cite{CFU} can be applied to $t\mapsto\log\left(t^{(a+1)}\frac{1+\lambda t}{t+\lambda}\right)$,
since as $a$ approaches $a_{0},$ its critical points 
\[
t_{\pm}=-\frac{\lambda^{2}(a+2)+a}{2\lambda(a+1)}\pm\sqrt{\left(\frac{\lambda^{2}(a+2)+a}{2\lambda(a+1)}\right)^{2}-1},
\]
coalesce to $t_{0}=1.$

\end{document}